\documentclass{amsart}
\usepackage{amssymb}
\usepackage{amsmath}
\usepackage{latexsym}
\usepackage{mathrsfs}
\usepackage{comment}
\usepackage{graphicx}
\usepackage{caption}
\usepackage{color}
\newtheorem{theorem}{Theorem}[section]
\newtheorem{thm}[theorem]{Theorem}

\newtheorem{defn}[theorem]{Definition}
\newtheorem{rem}[theorem]{Remark}
\newtheorem{lem}[theorem]{Lemma}
\newtheorem{corol}[theorem]{Corollary}

\catcode`@=11 \@addtoreset{equation}{section} \catcode`@=12
\newcommand{\R}{\mathbb{R}}
\newcommand{\Hmm}[1]{\leavevmode{\marginpar{\tiny%
$\hbox to 0mm{\hspace*{-0.5mm}$\leftarrow$\hss}%
\vcenter{\vrule depth 0.1mm height 0.1mm width \the\marginparwidth}%
\hbox to
0mm{\hss$\rightarrow$\hspace*{-0.5mm}}$\\\relax\raggedright #1}}}
\newcommand{\dx}{\,\mathrm{d}x}

\newcommand{\loc}{{\mathrm{loc}}}
\newcommand{\Div}{\mathrm{div}}
 \newcommand{\N}{\mathbb{N}}

\def\ga{\alpha}     \def\gb{\beta}       \def\gg{\gamma}
             
                         \def\vge{\varepsilon}
           
            \def\gl{\lambda}
\def\gm{\mu}

\def\Gw{\Omega}              
\begin{document}
\title[Hardy inequality]{$L^p$ Hardy inequality on $C^{1,\gamma}$ domains}
%
%----------Author 1
\author{Pier Domenico Lamberti}
\address{%
Dipartimento di Matematica ``Tullio Levi-Civita"\\
Universit\`a degli Studi di Padova\\
Via Trieste, 63\\
%P.O. Box 133\\
35126 Padova\\
Italy}
\email{lamberti@math.unipd.it}
%
%\thanks{This work was completed with the support of our
%\TeX-pert.}
%----------Author 2
\author{Yehuda Pinchover}
\address{Department of Mathematics\\
Technion - Israel Institute of Technology\\
Haifa 32000\\
Israel }
\email{pincho@technion.ac.il}
%----------classification, keywords, date
\subjclass[2010]{49R05, 35B09, 35J92.}
\keywords{Hardy inequality, Rayleigh quotient, spectral gap.}
%\date{November 05, 2013}
%----------additions
%
%%% ----------------------------------------------------------------------
%
\begin{abstract}
We consider the $L^p$ Hardy inequality involving the distance to the boundary of a domain in the $n$-dimensional Euclidean space with nonempty compact boundary.
We extend the validity of known existence and non-existence results, as well as the appropriate tight decay estimates for the corresponding minimizers, from the case of domains of class $C^2$ to the case of domains of class $C^{1,\gamma}$ with  $\gamma \in (0,1]$. We consider both bounded and
exterior domains.
The upper and lower estimates for the minimizers in the case of exterior domains and the corresponding related non-existence result seem to be new even for  $C^2$-domains.
\end{abstract}
%
%%% ----------------------------------------------------------------------
\maketitle
%%% ----------------------------------------------------------------------
%\tableofcontents
%%%%%%%%%%%%%%%%%%%%%%%%%%%%%%%%%%%%%%%%%%%%%%%%%%%%%%%%%%%%%%%%%%%%%%%%%%%%%%%%%%%%%%%%%%%%%%%%%%%%
%%%%%%%%%% INTRODUCTION
%%%%%%%%%%%%%%%%%%%%%%%%%%%%%%%%%%%%%%%%%%%%%%%%%%%%%%%%%%%%%%%%%%%%%%%%%%%%%%%%%%%%%%%%%%%%%%%%%%%%
%%%%%%%%%%
\section{Introduction}\label{sec0}
Let $\Omega$ be a domain in ${\mathbb{R}}^n$, $n\geq 2$ with nonempty boundary, and let $\delta (x)=d(x,\partial \Omega )$ denote the distance
of a point $x\in {\mathbb{R}}^n$ to the boundary of $\Omega$. Fix $p\in]1,\infty [$. We say that the {\em $L^p$ Hardy inequality} is satisfied in $\Omega $
if there exists $c>0$ such that
\begin{equation}
\label{mainhardy}
\int_{\Omega }|\nabla u|^p\dx \geq c \int_{\Omega }\frac{|u|^p}{\delta ^p}\dx  \qquad \mbox{ for all $u\in C^{\infty }_c(\Omega)$}.
\end{equation}
The {\em $L^p$ Hardy constant} of $\Omega$ is
the best constant for inequality (\ref{mainhardy}) which is denoted  here by  $H_p(\Omega )$. It is a classical result that goes back to Hardy himself (see for example \cite{baevke,permalkuf}) that if
$n=1$ and $\Omega $ is a bounded or unbounded interval, then the $L^p$ Hardy inequality holds and $H_{p}(\Omega)$ coincides with the widely known constant
$$
c_p=\biggl(\frac{p-1}{p}\biggr)^p.
$$
It is also well-known that   if $\Omega $ is bounded and has a sufficiently regular boundary in ${\mathbb{R}}^n$, then the $L^p$ Hardy inequality holds and $H_p(\Omega )\le c_p$ \cite{anc,mamipi}.  Moreover, if $\Omega$ is convex, and more generally if it is weakly mean convex, i.e., if $\Delta d\le 0$ in the distributional sense in $\Omega$, then $H_p(\Omega )=c_p $ \cite{barfilter, dam, mamipi}. On the other hand, it is also well-known (see for example \cite{baevke,permalkuf}) that if $\Omega ={\mathbb{R}}^n\setminus\{0\}$ and $p\ne n$, then the $L^p$-Hardy inequality holds and  $H_{p}(\Omega)$  coincides with the other widely  known constant
$$
c_{p,n}^*=\biggl|\frac{p-n}{p}\biggr|^p,
$$
which indicates  that  the $L^p$ Hardy inequality does not hold for ${\mathbb{R}}^n\setminus\{0\}$ if $p=n$ (for a short proof of this inequality see \cite{dev,defrpi}).   It also follows (see \cite{chwi,mamipi}) that if $\Omega $ is an exterior domain (i.e., an unbounded domain with nonempty compact boundary) with  sufficiently regular boundary and $p\ne n$, then the $L^p$ Hardy inequality holds with $H_p(\Omega )\le c_{p,n}$, where
$$
c_{p,n}=\min \{c_p, c_{p,n}^*  \}.
$$

The $L^p$ Hardy constant can be seen as the infimum of a Rayleigh quotient, namely
\begin{equation}\label{ray}
H_p(\Omega )=
\inf_{\substack{u\in \widetilde W^{1,p}(\Omega ) \\ u\ne 0}} \frac{\int_{\Omega} |\nabla u|^p\dx}{\int_{\Omega }\frac{|u|^p}{\delta ^p}\dx},
\end{equation}
 where
 $$  \widetilde W^{1,p}(\Omega ) := \{u\in W^{1,p}_{\loc}(\Omega )\mid   \|u\|_{L^p(\Omega ;\delta^{-p} )}+ \|\nabla u\|_{L^p(\Omega )}<\infty  \},$$
   and $ \|u\|_{L^p(\Omega ;\delta^{-p} )}:=(\int_{\Omega }|u|^p\delta^{-p}\dx)^{1/p}$ is the natural weighted $L^p$ norm associated with this problem.  Note that if $\Omega$ is a bounded domain with regular boundary, say of class  $C^1$,  then  $\widetilde W^{1,p}(\Omega ) = W^{1,p}_0(\Omega )$ (one can use the same argument as in  \cite[Appendix B]{mamipi}), while the two spaces do not coincide if, for example, $\Omega$ is an exterior domain and $p>n$, in which case the  first space  contains functions that are constant or even unbounded at infinity.

  It is important to note  that if the infimum for (\ref{ray})  is achieved at a function  $u$, then $u$  satisfies the corresponding Euler-Lagrange equation
 \begin{equation}\label{classiceq}
  -\Delta_pu - \frac{H_p(\Omega) }{\delta^p} {\mathcal I}_pu = 0,
   \end{equation}
in $\Omega$,  where $-\Delta_p v: =-\Div\big(|\nabla v|^{p-2}\nabla v\big)$ is the celebrated $p$-Laplace operator, and the operator ${\mathcal{I}}_p $ is defined by
${\mathcal I}_pv:= |v|^{p-2}v$. In this case, $H_p(\Omega)$ can be considered as the principal weighted eigenvalue of the $p$-Laplacian with respect to the Hardy weight, and $u$ is a corresponding principal eigenfunction. In particular, it  turns out that if the minimizer $u$ exists, then it is unique up to scalar multiples, and $u$ does not change its sign in $\Omega$.

We refer to   \cite{mamipi} for an introduction to this topic and to   \cite{anc, baevke, barfilter, bremar, chwi, dam, darev, davies, permalkuf, laso, marsha} and references therein for more information. We refer also to  \cite{avk, barlam, barlamisr, BarTer, BarTer2, dev, depi}   for recent developments in this subject.

The focus of the present paper is on the  problem of the  existence of  minimizers for (\ref{ray}). In the case of bounded domains of class $C^2$ this problem was solved
in \cite{mamipi, marsha}  where, among other results, it  was proved that a minimizer  exists for (\ref{ray}) if and only if  $H_p(\Omega )<c_p$.
In the case of exterior $C^2$-domains, it was proved in \cite{chwi} that if $H_p(\Omega)< c_{p,n}$,  then a minimizer exists for (\ref{ray}).
Importantly, in   \cite{chwi, mamipi, marsha} the assumption that $\Omega$ is of class $C^2$  is used in a substantial way, and weakening this
assumption seems highly nontrivial. Indeed, many arguments used in such papers are based on the well-known {\em tubular neighbourhood theorem} which allows to use tubular coordinates
near the boundary of  a domain of class $C^2$. Moreover, in \cite{mamipi, marsha} the assumption that $\Omega $ is of class $C^2$ is used also to guarantee
that the distance function $\delta$ is of class $C^2$ in a neighbourhood of the boundary, which in turn allows to use $\delta$ for the construction of
suitable   positive subsolutions and supersolutions of equation  (\ref{classiceq}). However,  the  tubular neighbourhood theorem does not hold  if $\Omega $ is of class  $C^{1,\gamma}$ with $0<\gamma <1$ and the distance function $\delta $  is not guaranteed to  be differentiable  near the boundary (the classical example is given by the parabolic open set $\Omega = \{(x,y)\in {\mathbb{R}}^2:\ y>|x|^{1+\gamma} \}$, in which case $\delta$ is not differentiable at all points $(0,y)$ of $\Omega$  close to $(0,0)$).

In the present paper, we prove that the existence and non-existence results in \cite{chwi, mamipi, marsha}  hold under the assumption that $\Omega$ is of class $C^{1,\gamma}$ with $\gamma \in (0,1]$, and we prove that the decay estimates for the minimizers  in \cite{mamipi, marsha} still hold.
Moreover, we provide decay and growth estimates  for the minimizers    also in the case of exterior domains  near the boundary and infinity.    Our approach  develops some ideas used in \cite{mamipi} for the  case $p=2$. In particular, we use the notion of {\em spectral gap} and {\em Agmon ground state},  and elaborate  the constructions of appropriate subsolutions  and supersolutions which replace those considered in \cite{mamipi, marsha}.  To do so, we first compute the so-called {\em Hardy constant at infinity}, i.e., the constant
\begin{multline}\label{lambdainfinity}
\lambda_{p,\infty }(\Omega )=\sup \left\{ \lambda \in {\mathbb{R}}\mid \exists  K\Subset \Omega\ {\rm and} \ u\in W^{1,p}_{\loc}(\Omega\setminus \bar K ) \ {\rm such \ that}    \right. \\  \left.  u>0 \
{\rm and }  -\Delta_pu - \frac{\lambda }{\delta^p}{\mathcal I}_pu \geq 0\  {\rm in }\  \Omega \setminus \bar K\right\},
\end{multline}
where we write $A \Subset O$ if $O$ is open, $\overline{A}$ is
compact and $\overline{A} \subset O$.

We prove that if $\Gw$ is a $C^{1,\gamma}$-domain with compact boundary, then $\lambda_{p,\infty}(\Omega)=c_p$ if $\Omega $
is bounded, and $\lambda_{p,\infty}(\Omega)=c_{p,n}$ if $\Omega $ is unbounded.

By a criterion   in  \cite[Lemma~4.6]{pin90}  (proved there only for the linear case),  it follows that  if $H_p(\Omega)<\lambda_{p,\infty }(\Omega )$, then the operator $-\Delta_p - \frac{H_p(\Omega) }{\delta^p} {\mathcal I}_p$ is critical in $\Gw$, which means that  it admits an Agmon ground state, i.e., a positive solution of equation  (\ref{classiceq}) in $\Gw$ which has minimal growth near $\partial \Gw$ and infinity.   We note that the quantity  $\lambda_{p,\infty }(\Omega ) -H_p(\Omega)$ is also referred to as the spectral gap, since  in the linear case ($p=2$),  $\lambda_{2,\infty }(\Omega )$ is the bottom of the essential spectrum of the operator $-\delta ^2 \Delta $, see \cite[\S~3]{mamipi}. Thus, the condition  $\lambda_{2,\infty }(\Omega ) -H_2(\Omega)>0$ implies that $H_2(\Omega )$ belongs to the discrete spectrum, and hence, it is an eigenvalue whose eigenfunction is the
 required minimizer.

The last step in the proof of the existence result, consists in proving that in the case of a spectral gap, the above mentioned Agmon ground state $u$ belongs to the space $\widetilde W^{1,p}(\Omega)$ and   this is done by constructing a supersolution $v$ to equation (\ref{classiceq}) which belongs to $L^p(\Omega ;\delta^{-p})$: indeed, $u$ being of minimal growth, it follows that $0<u\le C v $ near the boundary and  infinity, for some constant $C>0$, hence  $u\in L^p(\Omega ;\delta^{-p})$.

In a similar way, the non-existence of minimizers follows by a comparison principle proved in \cite{mamipi,marsha} combined with the construction of a suitable subsolution
 which does not belong to  $L^p(\Omega ;\delta^{-p})$.

  It is clear that one of the major ingredients of our arguments is the construction of  subsolutions and supersolutions  with the  appropriate growth  and this is used not only to provide the required estimates for the minimizers, but also for computing $\lambda_{p,\infty}(\Omega)$.
In \cite{mamipi, marsha} the construction of  subsolutions and supersolutions was done by using the so-called {\em Agmon trick}, namely, the
 subsolutions and supersolutions were given by functions of the type $\delta^{\alpha}+\delta ^{\beta}$ and  $\delta^{\alpha}-\delta ^{\beta}$, respectively, for suitable constants $\alpha, \beta >0$. As we have mentioned above, if $\Omega$ is of class $C^{1,\gamma}$ with $0< \gamma <1$ such functions
cannot be used. In this paper, we replace them by  functions of the form  $G^{\alpha}+G^{\beta }$ and $G^{\alpha }-G^{\beta }$,  where $G$ is a $p$-harmonic function
defined in a neighbourhood of the boundary and infinity. At infinity, the function $G$ is simply given by
$G(x)=|x|^\gb$ for an appropriate $\gb$.  Near the boundary of $\Omega$, the function $G$  can be  any positive $p$-harmonic function vanishing at $\partial \Omega$ (for example, one may consider the positive minimal Green function of the $p$-Laplacian, see Section~\ref{sec2} for details).  It is exactly at this point that the regularity  of $\Omega$ plays a crucial role. First, the assumption $\partial \Omega \in C^{1,\gamma}$  guarantees that the $p$-harmonic function $G$ is of class $C^{1,\tilde \gamma}$  up to  $\partial \Omega $ for some $\tilde \gamma \in (0, \gamma )$. Second, the same assumption  allows to use the Hopf lemma and to conclude that $\nabla G (x)\ne 0$ for all $x\in \partial \Omega $. The condition $\nabla G(x)\ne 0$ is
of fundamental importance for our analysis, since it allows  to control  the asymptotic behaviour of $\nabla G(x)/ G(x)$ as $x\to \partial \Omega$ in a precise way as it is explained in Lemma~\ref{general}.  We believe that Lemma~\ref{general} is of independent interest since it is proved without the use of the tubular neighbourhood theorem  and actually allows to bypass it.  We  note that the Hopf lemma  holds also if $\partial \Omega$ is of class $C^{1,{\rm Dini}}$ (see, \cite{miksha})
 and this allows to gain some generality as it is explained in Remark~\ref{remweak}.

Finally, we point out  that our results could be of  help in relaxing the boundary regularity assumptions of those statements in \cite{barlam, barlamisr}  the proofs of which require the existence of a minimizer for the variational problem (\ref{ray}).

The paper is organized as follows. In Section~\ref{sec2} we recall a number of notions concerning critical and subcritical operators and in particular, we reformulate  and generalize
the criterion \cite[Lemma~4.6]{pin90} in Lemma~\ref{pinchcriterion}. Section~\ref{sec3} is devoted to the construction of subsolutions and supersolutions, and in particular it contains the technical Lemma~\ref{general} which is applied to prove the key  lemmas~\ref{subsol} and \ref{subsolinf}.  In Section 4, we prove the
existence of minimizers and the corresponding decay and growth estimates, see theorems~\ref{mainbounded} and \ref{mainexterior}, and also Theorem~\ref{mainboundedweak} for a further relaxation of  the boundary conditions.  We conclude the paper in Section~\ref{sec5}, where we prove the lower estimates and the corresponding non-existence results, namely theorems~\ref{mainlower} and \ref{mainlower1}.
%%%%%%%%%%%%%%%%%%%%%%%%%%%
\section{Preliminaries}\label{sec2}
In this section we recall the notions of {\it positive minimal Green function} , {\it Agmon ground state}, and {\it subcritical and critical operators}. Moreover, we discuss a criterion for ensuring the
existence of  Agmon ground states for  equation (\ref{classiceq}). We refer to \cite{depi, defrpi, pin90, pinpsa,pintin} and references therein for details and proofs, and for  extensive discussions on this subject.

Fix $p\in ]1,\infty [$, and let $\Omega$ be a domain (i.e., an open connected set)  in ${\mathbb{R}}^n$,  where $n\geq 2$.  Let $V\in L^{\infty }_{\loc}(\Omega)$ (in fact, this assumption is not optimal, and we may assume that $V$ belongs to an appropriate local Morrey space, see \cite{pinpsa}). Consider the
operator
$$Q_V(u):=-\Delta_pu +V  |u|^{p-2}u,$$
and  the corresponding form ${\mathcal{Q}}$ defined by
$$
{\mathcal{Q}}_V(u,\varphi ):= \int_{\Omega }|\nabla u|^{p-2}\nabla u \cdot\nabla\varphi \dx +\int_{\Omega}V|u|^{p-2}u\varphi \dx,
$$
for all $u\in W^{1,p}_{\loc}(\Omega)$ and $\varphi\in C^{\infty}_c(\Omega)$.  As customary in the theory of quasilinear equations, we say that $u$ is a {\em (weak) solution} for  the equation $Q_V(v)=0$ in
$\Omega$ (or simply that $Q_V(u)=0$ in $\Omega$)   if $u\in W^{1,p}_{\loc}(\Omega )$ and $ {\mathcal {Q}}_V(u, \varphi )=0$ for all $\varphi \in C^{\infty }_c(\Omega)$. We also say that
$u$ is a {\em subsolution} (resp. {\em supersolution}) for the equation $Q_V(v)=0$ in $\Gw$ if $ {\mathcal {Q}}_V(u, \varphi )\le 0$ (resp. $ {\mathcal {Q}}_V(u, \varphi )\geq 0$)  for all $\varphi \in C^{\infty }_c(\Omega)$ with $\varphi \geq 0$; in these cases, we also simply write   $Q_V(u)\le 0$  (resp. $Q_V(u)\geq 0$) in $\Omega $.
Recall that by Allegretto-Piepenbrink theory \cite{pinpsa,pintin}, there exists a positive solution (or equivalently, a positive supersolution) for the equation $Q_V(u)=0$ in $\Gw$ if and only if $ {\mathcal {Q}}_V(\varphi , \varphi )\geq 0$ for all $\varphi \in C^{\infty }_c(\Omega)$, in which case the operator $Q_V$ is called non-negative in $\Gw$ (and we write $Q_V\geq 0$ in $\Gw$).
Obviously, the operator $Q_V$ is non-negative if $V\geq 0$.

We recall briefly some basic regularity results concerning solutions of the quasilinear equations appearing in the present paper. It is well known that if $V\in L_\loc^\infty(\Omega)$,  then solutions of the equation 
$$(-\Delta_p+V{\mathcal I}_p)u=0$$  
in $\Gw$ are in $C^{1,\alpha}(\Gw)$, and positive solutions satisfies the local Harnack inequality  (see for example \cite{gitr,HKM,MZ}).  Moreover, if the potential $V$ is bounded up to a $C^{1,\gg}$ portion of $\partial \Gw$, then solutions of the above equation are $C^{1,\alpha}$ up to this portion. Furthermore, the Hopf lemma holds for $p$-harmonic functions in $C^{1,\gg}$ domains and even under weaker assumptions (see Remark~\ref{remweak}).
%%%%%%%%%%%%%%%%%%%%%%%%%%%
\begin{defn}{\em Given a compact set $K$ contained in $\Omega$, we say that a positive solution $u$ to the equation $Q_V(u)=0$ in $\Omega \setminus K$ is a
{\em positive solution of minimal growth in a neighbourhood of infinity in $\Omega $} (briefly $u\in {\mathcal{M}}^{Q_V}_{\Omega , K}$) if for any bounded smooth  open set ${\mathcal {K}}$  with $K\subset {\mathcal{K}}\Subset \Omega  $ and any positive supersolution $v$ to the equation $Q_V(u)=0$ in $\Omega \setminus \bar {\mathcal{K}}$ we have that the condition $u\le v$ on $\partial {\mathcal{K}}$ implies that $u\le v$ on $\Omega \setminus \bar {\mathcal {K}}$.}
\end{defn}
We have the following theorem which includes the definitions of the notions mentioned above.
\begin{thm}[\cite{frpi,pinpsa,pintin}]\label{greenagmon} Assume that the operator $Q_V$ is non-negative in $\Omega$ and   fix $x_0\in \Omega $. Then there exists a  solution $u\in {\mathcal{M}}^{Q_V}_{\Omega , \{x_0\}}$  of $Q_V(u)=0$, and it  is unique up to a multiplicative constant.
Moreover, the following alternative holds:
\begin{itemize}
\item[(A)] either  $u$ has a singularity at the point $x_0$ with the following asymptotic behaviour
\begin{equation}
u(x)\sim \left\{
\begin{array}{ll}
|x-x_0|^{\frac{p-n}{p-1}} & {\rm if}\ 1<p<n,\\
-\log |x-x_0|& {\rm if}\ p=n,\\
1 & {\rm }\ p>n\, ,
\end{array}
  \right.
\end{equation}
 as $x\to x_0$, in which case $u$ is called a {\em positive minimal Green function} with pole at $x_0$  for $Q_V$ in $\Gw$, and the operator $Q_V$ is called {\em subcritical} in $\Gw$.

 \item[(B)] or $u$ is a global positive solution of the equation  $Q_V(v)=0$ in $\Omega $, in which case $u$ is called {\em Agmon ground state} for $Q_V$ and the operator $Q_V$ is called {\em critical}  in $\Gw$.
 \end{itemize}
 \end{thm}
Let $\Gw'$ be a subdomain of a domain $\Gw$ such that $\overline{\Gw'}\subset \Gw$. If $Q_{V}$ is non-negative in $\Gw$, then $Q_{V}$ is subcritical in $\Gw'$ \cite{pintin}.
Therefore,  if $\Omega$ is a domain with nonempty compact boundary and $\partial \Omega$ is sufficiently regular, then the $p$-Laplacian  (that is, $Q_V$ with $V=0$) is subcritical in $\Omega$, and hence, the corresponding function $u \in  {\mathcal{M}}^{Q_V}_{\Omega , \{x_0\}}$ is a positive minimal Green function.  Such a minimal Green function $G$ provides us with a positive $p$-harmonic function defined in a relative neighbourhood of $\partial \Omega$ which will be used in the sequel.
Importantly, if $\partial \Omega $ is of class $C^{1,\gamma}$ with $0<\gamma <1$, then $G$ is $C^{1,\ga}$ up to the boundary, $G(x)=0$  and $\nabla G(x)\ne 0$ for all $x\in \partial \Omega$ since the Hopf lemma
holds, see \cite{miksha}.

 In the case of linear elliptic equations, it was  stated and proved in   \cite[Lemma~4.6]{pin90} that the existence of a spectral gap implies the existence of an Agmon ground state.  The statement and  the proof in \cite{pin90} can be adapted to our case and for the convenience of the reader we indicate here how to do it.
 \begin{lem}\label{pinchcriterion} Let $\Omega $ be a domain in ${\mathbb{R}}^n$ such that the $L^p$ Hardy inequality holds, and let $V:=-H_p(\Omega)/\delta^p$. If the operator $Q_V$ has  a spectral gap, i.e.,  $H_p(\Omega)< \lambda_{p,\infty}(\Omega)$,  then  $Q_V$   is critical.
 \end{lem}
%%%%%%%%%%%%%%%%%%%%%%%%%%%
 \begin{proof}   Following   \cite[Lemma~4.6]{pin90}, we set
 $$S\!:=\!\{t\!\in\! {\mathbb{R}}\!\mid \! Q_{-t\delta^{-p}}\!\geq\! 0 \mbox{ in }\Gw \},\   S_{\infty}\!:=\!\{  t\!\in\! {\mathbb{R}}\!\mid\!  Q_{-t\delta^{-p}}\!\geq\! 0 \mbox{ in }  \Omega\! \setminus \!\bar K \mbox{ for some } K\!\Subset \Omega  \}.$$
Clearly, $S$ and $S_\infty$ are intervals, and since $Q_V$ has  a spectral gap, it follows that
$$S=\ ]\!-\infty , H_p(\Omega )] \ \varsubsetneq \ S_{\infty} \subseteq \ ]\!-\infty , \lambda_{\infty ,p}(\Omega)].$$
 For simplicity, we set $\lambda_0=H_p(\Omega )$.
Let $\lambda _1 \in S_{\infty }\setminus S$. 

{\bf Claim} There exists a nonzero non-negative potential $\mathcal{V}\in L^\infty(\Gw)$ with compact support in $\Omega$ such that $Q_{-\lambda_1\delta^{-p}+\mathcal{V} }\geq 0$ in $\Gw$.

Since $\gl_0<\gl_1<\gl_{\infty,p}$, there exists a smooth open  set $K_0\Subset  \Gw$ such that the equation  $Q_{-\lambda_1\delta^{-p}}(u)=0$ in $\Gw\setminus \bar{K_0}$ admits a positive solution. 

Fix a smooth open  set $K$ satisfying  $K_0\Subset K\Subset  \Gw$. We first show that there exists a positive solution $v$ of the equation  $Q_{-\lambda_1\delta^{-p}}(u)=0$ in $\Gw\setminus \bar{K}$ satisfying $v=0$ on $\partial K$.

To this end,  consider  a smooth exhaustion $\{\Gw_i\}_{i\in\N}$ of $\Gw$ by smooth relatively compact subdomains such that $x_0\in\Gw_1\setminus \bar{K}$    and such that 
$\bar{K}\subset \Omega_{i-1}\Subset \Omega_i $ for all $i>1$.  Let $v_{i}$ be the unique positive solution of the Dirichlet problem 
\[
 \left\{
 \begin{array}{ll}
   Q_{-\lambda_1\delta^{-p}}(u)=f_i & \mbox{in }\Gw_i\setminus \bar{K},
    \\[8pt]
     u=0 & \mbox{on }\partial (\Gw_i\setminus \bar{K}),
 \end{array}
 \right.
\]
where $f_i$ is  a nonzero nonnegative function in $C_c^\infty(\Gw_i\setminus \overline{\Gw_{i-1} })$      normalized in such a way that  $v_i(x_0)=1$. The existence and uniqueness of such a solution is guaranteed by \cite[Theorem~3.10]{pinpsa}  combined with the fact that $Q_{-\lambda_1\delta^{-p}}(u)=0$  admits a positive solution in $\Gw\setminus \bar{K_0}$. 

By the Harnack principle and elliptic regularity (see for example, \cite{pinpsa}) the sequence $\{v_i\}_{i\in\N}$ admits a subsequence converging locally uniformly to a positive solution $v$ of the equation  $Q_{-\lambda_1\delta^{-p}}(u)=0$ in $\Gw\setminus \bar{K}$ satisfying $v=0$ on $\partial K$.
Note that by classical regularity theory we  have that $v$ is of class $C^{1,\alpha }$ up to $\partial K$.    

Let  $K_1$  be an open set such that  $K\Subset K_1\Subset  \Gw$ and let 
$\min_{x\in \partial K_1} v(x)=m>0$. 
Let $\varepsilon>0$ be fixed in such a way that  $8\varepsilon < m$.  Let  $F$ be a $C^{2}$  function from $[0, +\infty [$ to $[0,+\infty [$ such that 
$F(t)=\varepsilon$ for all $0\le t \le 2\varepsilon $ and $F(t)= t$ for all $t\geq  4\varepsilon $  and such that $F'(t)\ne 0$ for all $t>2\varepsilon $. Assume also that 
$|F'(t)|^{p-2}F''(t)\to 0$ as $t\to 2\varepsilon$, hence the function $t\to |F'(t)|^{p-2}F''(t)$  (defined identically equal to zero on $ [0,2\varepsilon]$)  is continuous  on $[0, +\infty [$ (for this purpose, it is enough  for example that $F$ is chosen to be of the type $\varepsilon + (t-2\varepsilon)^{\beta }$ for all $t>2\varepsilon$ sufficiently close to $2\varepsilon$ and $\beta >\max\{p/(p-1),2\}$).        We set $\bar v(x) = F( v(x))$ for all   $x\in K_1\cap (\Omega \setminus {\bar  K})$.  By definition, it follows that there exists an open  neighborhood $U$ of $\partial K$ such that $\bar v (x) =\varepsilon $ for all $x\in U \cap (\Omega \setminus {\bar  K})$, and there exists an open neighborhood  $U_1$ of $\partial K_1$ such that $\bar v (x)=v(x)$ for all $x\in U_1\cap K_1$ . 
Thus $\bar v (x) $ can be extended continuously into the whole of $\Omega $ by setting $\bar v (x)=\varepsilon $ for all $x\in \bar K$  and $\bar v(x)= v(x)$ for all $x\in \Omega \setminus K_1$. 
 By   \cite[Lemma~2.10]{depi}  we have that 
 \begin{equation}
\label{superposition} 
-\Delta_p \bar v (x) =-|F'(v(x))|^{p-2}[(p-1)F''(v(x))|\nabla v(x)|^p+F'(v(x))\Delta _pv(x)    ] 
 \end{equation}
 for all $x\in K_1 \setminus \bar K$. 
By our assumptions on $F$ and $v$, it follows that $-\Delta_p \bar v (x) $ is a continuous function which  vanishes  
on $ U \cap (\Omega \setminus { \bar K})$ and equals $\lambda_1\delta^{-p} v$ in $U_1\cap K_1$. In particular, it makes sense to compute   $\Delta_p\bar v$ in  $\Omega $, and it turns out that 
 $\Delta_p\bar v=0 $  in $K$ and  $-\Delta_p\bar v= \lambda_1\delta^{-p} v$ in $\Omega\setminus \bar K_1$. 
 We can now define the potential ${\mathcal{V}}$ by setting
 $$
{\mathcal  {V}}= \frac{|Q_{-\lambda_1\delta^{-p}}(\bar v)   | }{\bar v ^{p-1}}.
 $$
  By construction, ${\mathcal  {V}}$ is a bounded function with compact support in $\Omega $ and 
 $\bar v$ is a positive supersolution of the equation $Q_{-\lambda_1\delta^{-p}+\mathcal{V}}(u)=0$ in $\Gw$. Hence,   $Q_{-\lambda_1\delta^{-p}+\mathcal{V} }\geq 0$ in $\Gw$,  and the Claim is proved.

\begin{comment}

{\color{red}
Indeed, let $K,K_1$ be smooth open  sets with  $K\Subset K_1\Subset \Gw$, and  $v$ a positive solution of the equation $Q_{-\lambda_1\delta^{-p}}(u)=0$ in $\Gw\setminus \bar K$. Let $\bar{v}$ be an extension of $v|_{\Gw\setminus K_1}$ as a positive $C^{1,\ga}$-function on $\Gw$.
Consider the solution $\hat  v$ on $K_1$ of the following Dirichlet problem:
$$-\Delta_p\hat v=\lambda_1\delta^{-p} \bar v ^{p-1} \quad \mbox{in } K_1, \quad \hat v=\bar v =v \quad \mbox{on } \partial K_1.$$
Then, we define the function $w$ by gluing together   $v$ and $\hat v$, i.e.,
$$w=\left\{
            \begin{array}{ll}
              v(x) & \hbox{if } x\in \Omega \setminus K_1,\\
              \hat v(x) & \hbox{if } x\in K_1.
            \end{array}
          \right.
$$
Let  $W:= |-Q_{-\lambda_1\delta^{-p}} w   |$. Then $W= |\lambda_1\delta^{-p}(\hat v^{p-1} - \bar v^{p-1}  ) | $   in $K_1$ and $W=0$ in $\Omega \setminus K_1$. Define the potential $\mathcal{V}$ as
$$\mathcal{V}=\left\{
            \begin{array}{ll}
              W/ w^{p-1} & \mbox{ in } K_1,\\
              0 & \hbox{ in } \Omega \setminus K_1.
            \end{array}
          \right.
$$
Then $\mathcal{V}\!\in\! L^\infty(\Gw)$, and $w$ is a positive supersolution of the equation $Q_{-\lambda_1\delta^{-p}+\mathcal{V}}(u)=0$ in $\Gw$. Hence,   $Q_{-\lambda_1\delta^{-p}+\mathcal{V} }\geq 0$ in $\Gw$.}

\end{comment}

   We set $\lambda_t=t\lambda_1+(1-t)\lambda_0$. By using \cite[Lemma~4.3]{depi} (see also \cite[Proposition~4.3]{pintin}), it follows that the set
   $$\{(t,s)\in [0,1]\times {\mathbb{R}}:\  Q_{-\lambda_t\delta^{-p}+s\mathcal{V} }\geq 0 \mbox{ in } \Gw\}$$  is a convex set. Hence, the function $\nu:[0,1]\to \R$  defined by
 $$\nu (t):=\min \{s\in {\mathbb{R}}:\  Q_{-\lambda_t\delta^{-p}+s\mathcal{V} }\geq 0 \mbox{ in } \Gw \}$$ is convex . Since $\mathcal{V}$ has compact support
 it follows by \cite[Proposition~4.19]{pinpsa} that $Q_{-\lambda_t\delta^{-p}+\nu (t)\mathcal{V} }$ is critical for all $t\in [0,1]$. We note that by definition $\nu (t)>0$ for all $t\in ]0,1]$, while
 $\nu (0)\le 0$. Since $\nu $ is convex, we must have $\nu (0)=0$, and hence $Q_{-\lambda_0\delta^{-p} }$ is critical. \end{proof}
%%%%%%%%%%%%%%%%%%%%%%%%%%%
\section{Construction of subsolutions and supersolutions}\label{sec3}
The proofs of our main theorems  are based on the  construction of   suitable subsolutions and supersolutions to equations of the form $-\Delta_pu-\lambda \delta^{-p} |u|^{p-2}u=0$,  which is carried out in this section. To do so, we need a number of preliminary results.

Recall our notation $ {\mathcal{I}}_p u = |u|^{p-2}u$.  By ${\mathbb{R}}^n\cup \{\infty \}$  we denote the standard  one-point compactification of ${\mathbb{R}}^n$ (note that in this paper the symbol $\infty$ will not be used with reference to  the one point compactification of a bounded domain $\Omega$, as often is done in the related literature). Finally, for  $\alpha \in [0,1]$  we  set
\begin{equation}\label{lambdaalpha}
\lambda_{\alpha}: = (p-1)\alpha ^{p-1}(1-\alpha).
\end{equation}
Observe that $\lambda_{\alpha }=c_p$ if $\alpha =(p-1)/p$, the function $\lambda_{\alpha }$ is increasing
with respect to  $\alpha \in[0, (p-1)/p]$ and decreasing for $\alpha \in [(p-1)/p,1]$.

The first part of the following lemma is taken from  \cite[Proposition~4.5]{depi}.
%%%%%%%%%%%%
\begin{lem}\label{devpinlem}  Let  $U$ be an open set in ${\mathbb{R}}^n$. Let  $G$ be  a positive  function  defined on  $U$ such that  $-\Delta_p G=0$ in $U$. Let
$W:=| \nabla G/ G |^p $. Then for every $\alpha \in (0,1)$ we have
\begin{equation}\label{devpinlem0}
(-\Delta_p-\lambda_{\alpha} W{\mathcal I}_p)G^{\alpha}=0,\ \ \  {\rm in}\ U.
\end{equation}

 Moreover, if $x_0\in \overline{ U}$, where the closure $\overline { U}$ of $U$ is taken in  ${\mathbb{R}}^n\cup \{\infty \}$,  and
 \begin{equation}\label{devpinlem1}
 \lim_{x\to x_0} \frac{| \nabla G(x)|}{ G(x)} d (x)=c
 \end{equation}
for some $c>0$,  where $d$ is a positive function defined in a relative punctured neighborhood of $x_0$. 
Then for every $\varepsilon >0$ there exists an open neighbourhood $U_{\varepsilon }$ of $x_0$ such that
\begin{equation}\label{asymlem1}
\biggl(-\Delta_p-\frac{c^p\lambda_{\alpha} -\varepsilon}{d^p}{\mathcal I}_p  \biggr)G^{\alpha}\geq 0\ \ \  {\rm in}\ \ \  (U_{\varepsilon}\cap U)\setminus \{x_0\}.
\end{equation}
\end{lem}
%%%%%%%%%%%%%%%%%%%%%%%
\begin{proof} For the proof of (\ref{devpinlem0}) we refer to \cite[Proposition~4.5]{depi}. In order to prove (\ref{asymlem1}) we note that
\begin{eqnarray}\lefteqn{
\biggl(-\Delta_p-\frac{c^p\lambda_{\alpha} -\varepsilon }{d^p}{\mathcal I}_p \biggr)G^{\alpha}=\biggl(-\Delta_p-\lambda_{\alpha } W {\mathcal I}_p  \biggr)G^{\alpha}}\nonumber \\  & &
+c^p \lambda_{\alpha } \left(\frac{ W}{c^p}-\frac{1}{d^p}  \right){\mathcal I}_p G^{\alpha}+\varepsilon \frac{{\mathcal I}_pG^{\alpha}}{d^p} =\left(\lambda_{\alpha} (Wd^p-c^p) +\varepsilon \right  )  \frac{{\mathcal I}_p
G^{\alpha}}{d^p} .
\label{asymlem2}\end{eqnarray}

By (\ref{devpinlem1}), it follows that there exists an open neighbourhood $U_{\varepsilon}$ of $x_0$ such that
$\lambda_{\alpha }(W(x)d (x)^p -c^p)  \geq -\varepsilon $ for all $x\in (U_{\varepsilon}\cap U)\setminus \{x_0\}$  which combined with (\ref{asymlem2}) yields (\ref{asymlem1}).  \end{proof}
%%%%%%%%%%%%%%%%%%%%%%%%%%%%%
The proof of the following lemma would be straightforward for open sets $\Omega$ of class $C^2$, in which case the tubular neighbourhood theorem holds and no boundary point can be approached by points from the cut locus of $\Omega$. However, assuming that $\Omega$ is of class $C^{1,\gamma}$ with  $0<\gamma <1$, or even just of class $C^1$ as we do here, requires a more detailed analysis.

As usual, by modulus of continuity of a real or vector-valued function $f$ defined on a subset $A$ of ${\mathbb{R}}^n$ we mean an increasing  function $\omega :[0,\infty [\to [0,\infty[$  such that $w(t)\to 0$ as $t\to 0$
and such that $|f(x)-f(y)|\le \omega (|x-y|)$ for all $x,y\in A$.
%%%%%%%%%%%%%%%%%%%%%%%%%%%%
\begin{lem}\label{general} Let $\Omega$ be an open set in  $\R^n$ of class $C^1$, $x_0\in\partial\Omega$ and $U$ be an open neighbourhood of $x_0$. Let $G\in C^1(\overline {\Omega \cap U})$ be a non-negative function such that
$G(x)=0$, $\nabla G (x)\ne 0$ for all $x\in U\cap \partial \Omega$. Then
\begin{equation}\label{general1}
\lim_{x\to x_0}\frac{|\nabla G(x)|}{G(x)} \delta (x)=1.
\end{equation}
Moreover, if $\omega$ is a  modulus of continuity of $\nabla G$ in a neighbourhood of $x_0$, then
\begin{equation}\label{general1,5}
\left| \frac{\nabla G(x)}{G(x)}\right|=\frac{1}{\delta  (x)}+\frac{ O(\omega (\delta (x)) )}{\delta (x)} \quad {\rm as}\ x\to x_0.
\end{equation}
\end{lem}
%%%%%%%%%%%%%%%%%%%%%%%%%%%%%%%%%%%%%
\begin{proof}  Since $\Omega $ is of class $C^1$, it can be represented locally around $x_0$ as the subgraph of a $C^1$ function. This means that there exists an open neighbourhood $B$ of $x_0$ and an isometry $R$ such that
$R(B)=\Pi_{i=1}^n]a_i, b_i[$ for  $a_i, b_i\in \R$ and
$R(\Omega \cap B)=\{  x\in \Pi_{i=1}^n]a_i, b_i[:   x_n< \varphi (x_1, \dots , x_{n-1}) \} $ where $\varphi $  is   a suitable
$C^1$ function from $\Pi_{i=1}^{n-1}[a_i, b_i]$ to $]a_n, b_n [$.  To shorten our notation, in the sequel we write $\bar x$ for $(x_1, \dots x_{n-1})$.   Moreover, we may assume directly that the isometry $R$ is the identity and that $B\Subset U$.  We now proceed dividing the proof in three steps.

\medskip

{\it Step 1.} We prove that there exists an open neighbourhood $\tilde B\subset B$ of $x_0$ and $c>0$ such that
\begin{equation}\label{general2}
c\delta (x) \le G(x) \le c^{-1}\delta (x)
\end{equation}
for all $x\in  \tilde B \cap\Gw$. Since $\nabla G$ is continuous up to $\partial \Omega$, $G$ vanishes on $\partial \Omega$ and $\nabla G$ does not vanish at any point of $\partial \Omega$, it follows that if $x\in \Omega \cap B$ is sufficiently close to $\partial \Omega$,
then $\frac{\partial G (x)}{\partial x_n}\ne 0$, hence there exists $c_1 >0 $ such that
\begin{equation}\label{general3}
c_1 \le
\left|   \frac{\partial G (x)}{\partial x_n} \right|
 \le c_1^{-1}
\end{equation}
for all $x\in \tilde B\cap \Gw$,  where $\tilde B$ is an open neighbourhood of $x_0$ with $\tilde B\subset B$.
Now, by the Lagrange's mean value theorem, we have
$G(\bar x,x_n)= \frac{\partial G (\bar x, \xi_x )}{\partial x_n} (x_n-\varphi (\bar x))$
where $\xi_x \in ]x_n,\varphi (\bar x)[$, hence
\begin{equation}\label{general4}
c_1(\varphi (\bar x) -x_n) \le G(x) \le c^{-1}_1( \varphi (\bar x)-x_n)
\end{equation}
for all $x\in \tilde B\cap\Gw$. By standard arguments and by possibly shrinking $\tilde B$, we have that there exists $c_2>0$ such that
 \begin{equation}\label{general5}
\delta (x)\le  \varphi (\bar x) -x_n\le c_2  \delta (x)
 \end{equation}
for all $x\in \tilde B\cap\Gw$, which combined with (\ref{general4}) yields (\ref{general2}).

\medskip

{\it Step 2.} Let $\omega $ be a modulus of continuity of $\nabla G$ on $\overline {\Omega \cap  B}$ as in the statement.  For every  $x\in \Omega$ we denote by $P(x)$ a point in  $\partial \Omega$ of minimal distance of $x$ from $\partial \Omega$,  which means that  $\delta (x)=|x-P(x)|$.  We prove that
\begin{equation}\label{general6}
G(x)=\nabla G(x)\cdot (x-P(x))+O(\omega (\delta (x)) )\delta (x)    \quad  {\rm as}\ x\to x_0.
\end{equation}
By the  Lagrange's mean value theorem applied to the function $$t\mapsto G(P(x)+t(x-P(x))),\quad  \mbox{where } t\in [0,1],$$  and $x$ is fixed in $\Omega\cap B$, we obtain
\begin{multline}\label{general7}
\lefteqn{G(x)=G(P(x))+\nabla G(P(x)+\eta_x(x-P(x)))\cdot (x-P(x))}\\
 =\nabla G(x)\cdot (x-P(x))+ \left(  \nabla G( P(x)+\eta_x(x-P(x)) ) -\nabla G(x) \right)\cdot (x-P(x)),
\end{multline}
for some $\eta _x\in ]0,1[$.   Then we have
\begin{multline}\label{general8} \lefteqn{
\left|
\left(  \nabla G( P(x)+\eta_x(x-P(x)) ) -\nabla G(x) \right)\cdot (x-P(x))
\right|  }    \\
 \le  \omega (  |  (\eta_x -1)(x-P(x)) |)|x-P(x) |\le \omega (\delta (x))\delta (x),
\end{multline}
 for all $x\in \Omega \cap B$. By  combining (\ref{general7}) and
 (\ref{general8}) we obtain (\ref{general6}).

\medskip

 {\it Step 3.} We note that
 \begin{equation}\label{general9}
 \lim_{x\to x_0}P(x)=x_0\ \ {\rm and}\ \ \frac{x-P(x)}{|x-P(x)|}= \nu (P(x)),
 \end{equation}
 where $\nu (P(x))$ is the unit inner normal to $\partial \Omega $ at the point $P(x)$. By (\ref{general6}) and the second equality in  (\ref{general9}) we have
 \begin{eqnarray}\label{general9,5}
 \frac{\nabla G(x)}{G(x)}\cdot \nu (P(x))= \frac{1}{\delta (x)}    +\frac{O(\omega (\delta (x)))  }{G(x)}\,.
 \end{eqnarray}
Consequently, by (\ref{general2}) and using the fact that $\omega (\delta (x))\to 0$ as $x\to x_0$, we deduce that
 \begin{equation}\label{general10}
 \lim_{x\to x_0}  \frac{\nabla G(x)}{G(x)}\cdot \nu (P(x)) \delta (x) =1.
 \end{equation}
Thus, by (\ref{general10})
 \begin{equation}
 \lim_{x\to x_0}\frac{|\nabla G(x)|}{G(x)}\delta (x)=\lim_{x\to x_0}\frac{|\nabla G(x)|  \nabla G(x)\cdot \nu (P(x)) }{ G(x)\nabla G(x)\cdot \nu (P(x))}  \delta (x)=\frac{ |\nabla G(x_0)| }{  \nabla G(x_0)\cdot \nu (x_0) }=1,
 \end{equation}
 where in the last equality we have used the fact that $\nabla G(x_0)=\nabla G(x_0) \cdot \nu (x_0)\nu (x_0)$ and $\nabla G(x_0) \cdot \nu (x_0)>0$ since $\nu $ points inwards. This completes the proof of (\ref{general1}).

\medskip

{\it Step 4.} For $x\in U\cap\Gw$ we consider an orthonormal  basis $$\{V_1(P(x)), \dots , V_{n-1}(P(x))\}$$ of the tangent hyperplane to $\partial \Omega$  at the point $P(x)$. Since $G$ vanishes on $U\cap \partial\Omega $ we have $\nabla G(P(x)) \cdot V_i(P(x))=0 $ for all $i=1,\dots , n-1$ hence
\begin{multline}\lefteqn{
 \nabla G(x) = \sum_{i=1}^{n-1}\nabla G(x) \cdot V_i(P(x))V_i(P(x))+ \nabla G(x)\cdot \nu (P(x)) \nu (P(x)) } \\
 = \sum_{i=1}^{n-1}(\nabla G(x)- \nabla G(P(x)))\cdot V_i(P(x))V_i(P(x))+ \nabla G(x)\cdot \nu (P(x)) \nu (P(x))\\
 = O( \omega (\delta (x)))+  \nabla G(x)\cdot \nu (P(x)) \nu (P(x)),
\end{multline}
which combined with (\ref{general2}) and (\ref{general9,5}) yields (\ref{general1,5}).  \end{proof}

We also need the following lemma which represents a special case of a general statement proved in \cite[Lemma~2.10]{depi}. Formula (\ref{plapsum1}) has to be understood in the distributional sense.

\begin{lem}\label{plapsum}  Let  $U$ be an open set in ${\mathbb{R}}^n$, and let $G$ be  a positive function of class $C^{1}(U)$. Then for all $\alpha , \beta >0$ we have
\begin{eqnarray}\label{plapsum1}\lefteqn{
\Delta_p (G^{\alpha}\pm G^{\beta})=\left|\alpha G^{\alpha -1}\pm \beta G^{\beta -1} \right|^{p-2}\biggl[\left(\alpha G^{\alpha -1}\pm \beta G^{\beta -1} \right) \Delta_pG \biggr.} \nonumber \\
& &\qquad\qquad\quad + (p-1)|\nabla G |^{p}\left[ (\alpha^2-\alpha)G^{\alpha -2  } \pm (\beta^2-\beta )G^{\beta -2  }\right] \biggr] .
\end{eqnarray}
\end{lem}
We are now ready to prove the following theorem which guarantees the existence of the above mentioned subsolutions and supersolutions
in a neighbourhood of a compact boundary.
\begin{lem}\label{subsol}
Let $\Omega $ be a domain in ${\mathbb{R}}^n$ with nonempty compact boundary  of class $C^{1}$ and
$U$ be an open neighbourhood of $\partial \Omega$. Let $\gamma \in ]0,1]$ and   $G\in C^{1,\gamma }(\overline {\Omega \cap U})$ be a positive function such that $\Delta_pG=0$ in $\Omega \cap U$ and $G(x)=0$ for all $x\in \partial \Omega$.  Let $\alpha , \beta \in (0,1)$ be such that $(p-1)/p\le \alpha <\beta  <\alpha  +\gamma  $.
  Then there exists an open neighbourhood ${\mathcal{U}}$ of $\partial \Omega$, ${\mathcal{U}}\subset U$, such that
 the functions $G^{\alpha}+G^{\beta}$  and $G^{\alpha}-G^{\beta}$   are a subsolution and a supersolution, respectively, for the equation $-\Delta_pv=\lambda_{\alpha }{\mathcal{I}}_pv /\delta^p$ in $\Omega \cap {\mathcal{U}}$, where $\lambda_{\alpha }=(p-1)\alpha ^{p-1}(1-\alpha)$. Moreover, ${\mathcal{U}}$ can be chosen to be independent  of  small perturbations of $\alpha$ and $\beta$.
\end{lem}
%%%%%%%%%%%%
\begin{proof} First we consider the case of the subsolution. By Lemma~\ref{plapsum}
 and   (\ref{general1,5}), it follows that
\begin{multline}\label{subsol1}
-\Delta_p (G^{\alpha}+G^{\beta})=\left(\alpha G^{\alpha-1}+\beta G^{\beta -1} \right)^{p-2}\left(\frac{\lambda_{\alpha}}{\alpha^{p-2}}G^{\alpha -2}+\frac{\lambda_{\beta}}{\beta^{p-2}}G^{\beta -2}  \right)|\nabla G|^p \\[2mm]
=G^{\alpha (p-1) }\left(\alpha +\beta G^{\beta -\alpha}\right)^{p-2}\left(\frac{\lambda_{\alpha}}{\alpha^{p-2}}+\frac{\lambda_{\beta}}{\beta^{p-2}}  G^{\beta-\alpha }\right)\left|\frac{\nabla G}{G}\right|^p \\[2mm]
\le G^{\alpha (p-1)}(\alpha +\beta G^{\beta -\alpha})^{p-2}\left(\frac{\lambda_{\alpha}}{\alpha^{p-2}}+ \frac{\lambda_{\beta}}{\beta^{p-2}} G^{\beta-\alpha }\right)\left(\frac{1}{\delta ^p}+O(\delta ^{\gamma -p})\right).
\end{multline}

By (\ref{subsol1}), in order to guarantee that $G^{\alpha}+G^{\beta}$ is a subsolution as required in the statement,  it suffices to  impose the condition
\begin{equation*}
G^{\alpha (p-1) }\!\left(\alpha +\beta G^{\beta -\alpha}\right)^{p-2}\!\!\left(\frac{\lambda_{\alpha}}{\alpha^{p-2}} +\frac{\lambda_{\beta}}{\beta^{p-2}}G^{\beta-\alpha }\right)\!\!\left(\frac{1}{\delta ^p}+O(\delta ^{\gamma -p})\right)\!\!\leq\! \frac{\lambda_{\alpha}}{\delta^p} (G^{\alpha}+G^{\beta })^{p-1}
\end{equation*}
which can written in the form
\begin{equation}\label{subsol2}
\left(\alpha +\beta G^{\beta -\alpha}\right)^{p-2}\left(\frac{\lambda_{\alpha}}{\alpha^{p-2}}+\frac{\lambda_{\beta}}{\beta^{p-2}}G^{\beta-\alpha }\right)
\left(1+O(\delta ^{\gamma })\right)\leq \lambda_{\alpha} (1+G^{\beta -\alpha })^{p-1}.
\end{equation}
Since $G^{\beta -\alpha}=0$ on $\partial \Omega$, by expanding both sides of (\ref{subsol2}) in $G^{\beta -\alpha}$ up to the first order,
inequality (\ref{subsol2}) can also be written in the form
\begin{equation}\label{subsol2bis}
(\lambda_{\alpha} +AG^{\beta -\alpha} +o(G^{\beta -\alpha }))\left(1+O(\delta ^{\gamma })\right)\le \lambda_{\alpha}(1 +(p-1)G^{\beta -\alpha }+o(G^{\beta -\alpha})),
\end{equation}
where
$$A:=(p-2)\lambda_{\alpha}\beta/\alpha+\lambda_{\beta}\alpha^{p-2}/\beta^{p-2}.$$
Note that since $G(x)$ is asymptotic to $\delta (x)$ as $x\to \partial \Omega $ and $\beta -\alpha < \gamma $, we have that  $\delta (x)^{\gamma }  /G(x)^{\beta -\alpha }\to 0$ as $x\to \partial \Omega $. Moreover, by a direct   computation and by using  condition $(p-1)/p \le \alpha <\beta$, one can easily verify that $A<(p-1)\lambda_{\alpha}$. Thus, passing to the limit as $x\to \partial \Omega$ in both sides of (\ref{subsol2bis}), one can see  that condition (\ref{subsol2bis})  is satisfied in $\Omega \cap {\mathcal{U}}$, where ${\mathcal{U}}$ is a suitable neighbourhood of $\partial\Omega$ which can be chosen to be independent of $\alpha $ and  $\beta$, if $\alpha $ and $\beta$ are as in the statement and belong to small neighbourhoods of two fixed parameters
$\alpha_0$, $\beta_0$ satisfying the conditions $(p-1)/p\le \alpha_0< \beta_0$.

\medskip

We now consider the case of the supersolution. Proceeding as above, we see that in order to guarantee that $G^{\alpha}-G^{\beta}$ is a positive supersolution as required in the statement, we clearly may first take a small neighbourhood  ${\mathcal{U}}_1$ of $\partial \Gw$ such that $G^{\alpha}-G^{\beta}$ is positive in $\Omega \cap {\mathcal{U}}_1$.
  So,  it suffices to impose the condition
\begin{equation*}
G^{\alpha (p-1) }\!\!\left|\alpha -\beta G^{\beta -\alpha}\right|^{p-2}\!\!\left(\frac{\lambda_{\alpha}}{\alpha^{p-2}} -\frac{\lambda_{\beta}}{\beta^{p-2}}G^{\beta-\alpha }\right)\!\!\left(\frac{1}{\delta ^p}-O(\delta ^{\gamma -p})  \right)\geq \frac{\lambda_{\alpha}}{\delta^p} (G^{\alpha}-G^{\beta })^{p-1}
\end{equation*}
in $\Omega \cap {\mathcal{U}_2}$, where ${\mathcal{U}_2}$ is a smaller neighbourhood of $\partial \Gw$. The latter inequality can be written in the form
\begin{equation}
(\lambda_{\alpha} -AG^{\beta -\alpha} +o(G^{\alpha -\beta }))\left(1-O(\delta ^{\gamma })\right)\geq \lambda_{\alpha}(1 -(p-1)G^{\beta -\alpha }+o(G^{\alpha -\beta})),
\end{equation}
where $A$ is the same constant defined above. Again, since $A <(p-1) \lambda_{\alpha}$ we easily deduce  as in the case of the subsolution the desired assertion. \end{proof}

\medskip

We now construct  sub- and super-solutions near $\infty$ for the operator $$-\Delta_p-\lambda_{\alpha } \left| \frac{p-n}{p-1} \right|^p  \frac{{\mathcal{I}}_p}{\delta^p}$$ on an unbounded domain $\Gw$ with compact boundary. Recall that if $p=n$, then for such a domain $H_p(\Gw)=0$.  So, for our purpose, we need to consider only the case where $p\neq n$.
%%%%%%%%%%%%%%%%%%%%%%%%%%%%%%%%%
\begin{lem}\label{subsolinf} Let $\Omega $ be an unbounded domain in ${\mathbb{R}}^n$ with nonempty compact boundary. Let $G$ be the function defined in ${\mathbb{R}}^n\setminus \{0\}$ by $G(x):=|x|^{\frac{p-n}{p-1}}$ for all $x\in {\mathbb{R}}^n\setminus \{0\}$.
Then the following statements hold:
\begin{itemize}
\item[(i)] If $p<n$ and  $\alpha , \beta \in (0,1)$ are such that
$$\frac{p-1}{p}\le \alpha <\beta  <\alpha  +\frac{p-1}{n-p}  \,,$$
then  there exists  $M>0$ such that
 the functions $G^{\alpha}+G^{\beta}$, $G^{\alpha}-G^{\beta}$ are a subsolution and a  supersolution, respectively, for the equation $-\Delta_pv=\lambda_{\alpha } \left| \frac{p-n}{p-1} \right|^p  {\mathcal{I}}_pv /\delta^p$ on $\{x\in {\mathbb{R}}^n :\ |x|>M \}$.
 \item[(ii)]  If $p>n$ and  $\alpha , \beta \in (0,1)$ are such that $  \beta <\alpha  \le (p-1)/p$,
  then  there exists  $M>0$ such that
 the functions $G^{\alpha}+G^{\beta}$, $G^{\alpha}-G^{\beta}$ are a subsolution and a supersolution respectively, for the equation $-\Delta_pv=\lambda_{\alpha }  \left| \frac{p-n}{p-1} \right|^p {\mathcal{I}}_pv /\delta^p$ on $\{x\in {\mathbb{R}}^n :\ |x|>M \}$.
 \end{itemize}
\end{lem}
%%%%%%%%%%%%%%%%%%%%%%%
\begin{proof} By Lemma~\ref{plapsum} it follows that
\begin{multline*}
-\Delta_p (G^{\alpha}\pm G^{\beta}) =\left|\alpha G^{\alpha-1}\pm \beta G^{\beta -1} \right|^{p-2}\left(\frac{\lambda_{\alpha}}{\alpha^{p-2}}G^{\alpha -2}\pm \frac{\lambda_{\beta}}{\beta^{p-2}}G^{\beta -2}  \right)|\nabla G|^p \\
=G^{\alpha (p-1) }\left|\alpha \pm \beta G^{\beta -\alpha}\right|^{p-2}\left(\frac{\lambda_{\alpha}}{\alpha^{p-2}}\pm \frac{\lambda_{\beta}}{\beta^{p-2}}  G^{\beta-\alpha }\right)\left|\frac{\nabla G}{G}\right|^p \\
=G^{\alpha (p-1) }\left|\alpha \pm \beta G^{\beta -\alpha}\right|^{p-2}\left(\frac{\lambda_{\alpha}}{\alpha^{p-2}}\pm \frac{\lambda_{\beta}}{\beta^{p-2}}  G^{\beta-\alpha }\right)
\left|   \frac{p-n}{p-1}  \right|^p\frac{1}{|x|^p}\,.
\end{multline*}
Since $\partial \Omega$ is compact, it follows that
$|\delta^{-p}-|x|^{-p} |\le O(\delta^{-1})\delta^{-p} $ as $|x|\to \infty$.  Thus, in order to verify that $G^{\alpha}+G^{\beta}$ is a subsolution as required in the statement, it suffices to impose the condition
\begin{equation}\label{subsol1inf}
(\alpha + \beta G^{\beta -\alpha})^{p-2}\left(\frac{\lambda_{\alpha}}{\alpha^{p-2}}+ \frac{\lambda_{\beta}}{\beta^{p-2}}  G^{\beta-\alpha }\right)
\left(  1+O(\delta^{-1})\right)
 \le \lambda_{\alpha} (1+G^{\beta -\alpha })^{p-1}.
\end{equation}
Similarly, in order to guarantee that $G^{\alpha}-G^{\beta}$ is a supersolution as required in the statement, it suffices to impose the condition
\begin{equation}\label{subsol1infbis}
\left|\alpha - \beta G^{\beta -\alpha}\right|^{p-2}\!\!\left(\frac{\lambda_{\alpha}}{\alpha^{p-2}}- \frac{\lambda_{\beta}}{\beta^{p-2}}  G^{\beta-\alpha }\right)
\left(  1 -O(\delta^{-1})\right)
 \geq \lambda_{\alpha} |1-G^{\beta -\alpha }|^{p-2}(1-G^{\beta -\alpha }).
\end{equation}

By assumptions, in both cases $p<n$ and $n<p$, we have that $ G^{\beta -\alpha}(x)\to 0$ as $|x|\to \infty $. Thus, condition (\ref{subsol1inf}) can be written as
\begin{equation}\label{subsol1infre}
(\lambda_{\alpha} +AG^{\beta -\alpha} +o(G^{\beta -\alpha }))\left(1+ O(\delta ^{-1})\right)
\le \lambda_{\alpha}(1 +(p-1)G^{\beta -\alpha }+o(G^{\beta -\alpha})),
\end{equation}
while condition  (\ref{subsol1infbis}) can be written as
\begin{equation}\label{subsol1infrebis}
(\lambda_{\alpha} -AG^{\beta -\alpha} +o(G^{\beta -\alpha }))\left(1- O(\delta ^{-1}) \right)
\geq \lambda_{\alpha}(1 -(p-1)G^{\beta -\alpha }+o(G^{\beta -\alpha})),
\end{equation}
where in both cases $A=(p-2)\lambda_{\alpha }\beta/\alpha+\lambda_{\beta}\alpha^{p-2}/\beta^{p-2}$ is the same constant appearing in the proof of Lemma~\ref{subsol}. As it was noted in the proof of Lemma~\ref{subsol}, if $(p-1)/p\le \alpha <\beta $, then $A<(p-1)\lambda_{\alpha}$. However, it can be easily seen that $A<(p-1)\lambda_{\alpha}$ also if $0<\beta <\alpha \le (p-1)/p$. It follows that in order to verify the validity of conditions (\ref{subsol1infre}) and (\ref{subsol1infrebis}) for $|x|$ large enough, it suffices to verify that  $O(\delta ^{-1}) G^{\alpha -\beta }=o(1)$  as $|x|\to \infty$. This condition is satisfied because $G(x)$ is asymptotic to $\delta(x)^{\frac{p-n}{p-1}}$ as $|x|\to \infty$ and
   $|\alpha -\beta |<|(p-1)/(p-n)| $. \end{proof}
%%%%%%%%%%%%%%%%%%%%%%%%%%%%%%%%%%%%%%%%%%%%
\section{Upper bounds and existence of minimizers}\label{sec4}
\noindent Using the results of the previous section, we can prove the following existence result for bounded domains. Note that,  assuming that $\Omega $ is of class $C^{1,\gamma}$ as we do here, would allow to
skip a few steps in our proof. However, we prefer to write down more details which explain how our method could be adapted to more general situations as described in Theorem~\ref{mainboundedweak}, see
Remark~\ref{remweak} below.
%%%%%%%%%%%%%%%%%%%%%%%%%%%
\begin{thm}\label{mainbounded} Let $\Omega $ be a bounded domain in $\R^n$ of class $C^{1,\gamma}$ with $\gamma \in ]0,1]$. Then
$\lambda_{p,\infty}(\Omega)=c_p$. Moreover, if  $H_p(\Omega )< c_p$, then there exists a positive minimizer $u\in W^{1,p}_0(\Omega )$ for (\ref{ray}). In particular, if $\alpha \in \,](p-1)/p, 1[$ is such that $\lambda_{\alpha}=H_p(\Omega )$,  then \begin{equation}
\label{mainboundedeq}
0<u(x)\le C\delta^{\alpha} (x) \qquad \forall x\in \Gw.
\end{equation}
\end{thm}
%%%%%%%%%%%%%%%%%%
\begin{proof} Let  $\tilde x_0\in \Omega$ and  $G$ be the positive minimal Green function in $\Gw$ of the $p$-Laplacian with pole at $\tilde x_0$. Recall that
since $\Omega $ is of class $C^{1,\gamma}$, then $G$ is of class $C^{1,\tilde \gamma }$ away from $\tilde x_0$ and up to $\partial \Omega$, for some $\tilde \gamma \in (0,\gamma )$, and $G(x)=0$,  $\nabla G(x)\ne 0$ for all $x\in \partial \Omega$ by the Hopf lemma (see Section~\ref{sec2}).
Thus, $G$ satisfies equality (\ref{general1}) for all $x_0\in \partial \Omega$. In particular, choosing  $\alpha =(p-1)/p$ in  Lemma~\ref{devpinlem}, we have that
$G^{\alpha }$ satisfies (\ref{asymlem1}) with $\lambda_{\alpha} =c_p$ and $c=1$.

Since  $\partial \Omega$ is compact, it follows that $G^{(p-1)/p}$ is a supersolution to the equation $-\Delta_pv-(c_p-\varepsilon  ){\mathcal{I}}_pv/\delta^p=0  $ in a relative neighbourhood of $\partial \Omega$. By passing to the limit as $\varepsilon \to 0$ and using definition (\ref{lambdainfinity}), we get that $\lambda_{p,\infty}(\Omega)\geq c_p$. 

\medskip

On the other hand, since $\Omega $ is of class  $C^{1}$, any point at the boundary has a tangent hyperplane, hence locally around any  fixed point at the boundary it is possible to apply  the same argument of \cite[Theorem~5]{mamipi} and conclude that $\lambda_{p,\infty}(\Omega)\le c_p$. 

More precisely,  
let $P\in \partial \Omega$ be fixed and $\Pi$ be the tangent hyperplane at $\partial\Omega$ in $P$. We claim  that condition (2.2) in  \cite[Theorem~5]{mamipi} is satisfied, that is, for all $x\in \Omega$ in a suitable neighborhood of $P$  we have
\begin{equation}
\label{tangent}
|d(x, \Pi)-\delta (x)|\le o(1)d(x,P)
\end{equation}
where $o(1) $ is a quantity which tends to zero as $x\to P$. To prove \eqref{tangent} we argue as follows.
 Since $\Omega$ is of class  $C^1$, we can  assume without loss of generality that $P=0$ (the origin of the coordinate system), $\Pi =\{x\in {\mathbb{R}}^n:\ x_n=0  \}$ and that     there exists an open neighborhood $U$ of $P$ such that 
$$
\Omega\cap U =\{ (\bar x, x_n)\in {\mathbb{R}}^n:\ \bar x\in \Pi_{i=1}^{n-1}]a_i, b_i[,\ a_n<x_n<g(\bar x)     \},
$$
for suitable real numbers $a_i,b_i$, where $g$ is a function of class $C^1(\Pi_{i=1}^{n-1}]a_i, b_i[)$ such that $g(0)=0$ and $\nabla g (0)=0$. 
Given  $x=(\bar x, x_n)\in \Omega \cap U$, we set $r_x=d(x, P)=|x|$ and  $L_x=\sup_{z\in B(x, 2r_x)\cap \Omega }|\nabla g(\bar z)|$. We note that $L_x$ is well-defined for $x$ sufficiently close to $P$ and that $\delta (x)=\inf_{y\in  B(x, 2r_x)\cap \Omega  }d(x,  (\bar y, g(\bar y) )$. 
Thus, for any $y\in  B(x, 2r_x)\cap  \Omega $ we have
$$
|x_n-g(\bar x)|\le |x_n-g(\bar y)|+|g(\bar y)-g(\bar x)  |\le  |x_n-g(\bar y)|+L_x|\bar x-\bar y|\le (1+L_x)d(x,  (\bar y, g(\bar y) )
$$
hence
\begin{equation}
\label{tangent1}
|x_n-g(\bar x)|\le  (1+L_x)\inf_{y\in  B(x, 2r_x)\cap \Omega  }d(x,  (\bar y, g(\bar y) )=(1+L_x)\delta(x).
\end{equation}
It follows from \eqref{tangent1} that 
$$
|x_n|\le |x_n-g(\bar x)|+|g(\bar x)|\le (1+L_x)\delta (x)+L_x|\bar x|\le (1+L_x)\delta (x)+L_xd(x, P)
$$
hence 
\begin{equation}
\label{tangent2}
|x_n|-\delta (x)\le L_x\delta (x)+L_xd(x, P)\le 2L_xd(x, P).
\end{equation}
On the other hand, 
\begin{equation}
\label{tangent3}
\delta (x)\le |x_n-g(\bar x)| \le |x_n|+L_x|\bar x|\le |x_n|+L_xd(x, P).
\end{equation}
In conclusion, combining \eqref{tangent2} and \eqref{tangent3} we get
\begin{equation}
\label{tangent4} 
||x_n|-\delta (x))  |\le 2L_xd(x, P)
\end{equation}
where $L_x\to 0$ as $x\to P$ since $\nabla g$ is continuous. Thus condition \eqref{tangent} is satisfied. 

Now, condition \eqref{tangent} (together with the fact that a segment perpendicular to $\Pi$ is contained in $\Omega$, which is clearly satisfied in our case)
is used in \cite[Theorem~5]{mamipi}  to prove that  for any open neighbourhood  $V$ of $P$ and  any $\varepsilon>0$ there exists a function $\varphi\in C^{\infty }_c(V\cap \Omega ) $ such that  
$$
 \frac{\int_{\Omega} |\nabla \varphi|^p\dx}{\int_{\Omega }\frac{|\varphi |^p}{\delta ^p}\dx}\le (1+\varepsilon)(c_p+\varepsilon)
$$
which allows us to conclude that the admissible numbers  $\lambda$ in \eqref{lambdainfinity} satisfy $\lambda \le  (1+\epsilon)(c_p+\varepsilon)$   for any $\varepsilon >0$  (recall the Allegretto-Piepenbrink theory mentioned at the beginning of Section~\ref{sec2}). Hence,   $\lambda_{p,\infty}(\Omega)\le c_p$.
This proves that $\lambda_{p,\infty}(\Omega)= c_p$.

\medskip

We assume now that $H_p(\Omega )< c_p$ and prove the existence of a minimizer for (\ref{ray}). First of all we note that since $H_p(\Omega)<\lambda_{p,\infty }(\Omega)$, Lemma~\ref{pinchcriterion} implies that the positive function of minimal growth $u \in {\mathcal{M}}^{Q_{-H_p(\Omega )\delta^{-p}}}_{\Omega , \{x_0\}}$ is an Agmon ground state. 

\medskip

We now prove that $u\in L^p(\Omega ; \delta^{-p})$.
  Since $\lambda_{\alpha} = c_p $ if $\alpha =(p-1)/p$ and $H_p(\Omega )< c_p$, we can choose $\tilde \alpha >(p-1)/p$  close enough to $(p-1)/p$ so that $\lambda_{\tilde \alpha} >H_p(\Omega )$. Note that this choice of $\tilde \alpha$ implies that $G^{\tilde \alpha}\in L^p(\Omega, \delta^{-p})$.
As above, using (\ref{asymlem1}) and the compactness of $\partial\Omega$ it follows that the function $G^{\tilde \alpha }$ is a supersolution to the equation $-\Delta_pv-(   \lambda_{\tilde \alpha}  -\varepsilon  ){\mathcal{I}}_pv/\delta^p=0  $ in a relative neighbourhood of $\partial \Omega$. Hence, in such a neighbourhood
\begin{eqnarray}
\left(-\Delta_p-H_p(\Omega )\frac{{\mathcal{I}}_p}{\delta^p}\right)G^{\tilde \alpha}\geq   \left(-\Delta_p-(\lambda_{\tilde \alpha } -\varepsilon )\frac{{\mathcal{I}}_p}{\delta^p}\right)G^{\tilde \alpha}  \geq 0,
\end{eqnarray}
provided that $\varepsilon >0 $ is small enough to guarantee that $H_p(\Omega ) \le \lambda_{\tilde \alpha} -\varepsilon $.  Thus, $G^{\tilde \alpha}$ is a positive supersolution to the equation $-\Delta_pv-H_p(\Omega )\frac{{\mathcal{I}}_pv}{\delta^p}=0$ in a relative neighbourhood of $\partial \Omega $. Therefore,  the ground state $u$ satisfies the condition $0<u\le k G^{\tilde \alpha}$ in a relative neighbourhood of $\partial \Omega$ for a suitable positive constant $k$. This implies that $u\in L^p(\Omega, \delta^{-p})$.

\medskip

We now prove that $\nabla u\in L^p(\Omega )$. Note that since $u\le k G^{\tilde \alpha}$ in a relative neighbourhood of $\partial \Omega$, we have that
$u(x)\to 0$ as $x\to \partial \Omega$, hence $u$ is continuous up to the boundary of $\Omega$. Then we use a standard truncation argument as follows. For any $\varepsilon>0$ we consider the real-valued  function $F_{\varepsilon }$ defined on $[0, \infty [$ by setting $F_{\varepsilon}(x)=0$ if $0\le x<\varepsilon /2$,
$F_{\varepsilon }(x)= 2x-\varepsilon $ if  $\varepsilon /2<x<\varepsilon $, $F_{\varepsilon }(x)=x$ if $x\geq \varepsilon$. Moreover, we set $u_{\varepsilon }=F_{\varepsilon }\circ u$. Since  $u_{\varepsilon }$ has compact support in $\Omega$, it can be used as a test function in the weak formulation of the problem solved by $u$, namely
\begin{equation}\label{agmonweak}
\int_{\Omega}|\nabla u|^{p-2}\nabla u\nabla \varphi \dx= H_p(\Omega )\int_{\Omega }\frac{|u|^{p-2}u\varphi }{\delta ^p}\dx,
\end{equation}
where one can see by a standard approximation argument that it is possible to choose not only test functions $\varphi \in C^{\infty }_c(\Omega )$ but
 also functions in $W^{1,p }(\Omega )$ with compact support. Plugging $u_{\varepsilon }$ in (\ref{agmonweak}) we get
\begin{equation}
\label{agmonweak1}
\int_{\{x\in \Omega:\ u(x)\geq \varepsilon  \}}|\nabla u|^{p} \dx + 2\int_{\{x\in \Omega:\ \varepsilon /2 <u(x)< \varepsilon  \}}|\nabla u|^{p} \dx =
H_p(\Omega )\int_{\Omega }\frac{|u|^{p-2}uu_{\varepsilon } }{\delta ^p}\dx
\end{equation}
which in particular yields
\begin{equation}
\label{agmonweak2}
\int_{\{x\in \Omega:\ u(x)\geq \varepsilon  \}}|\nabla u|^{p} \dx \le H_p(\Omega )\int_{\Omega }\frac{|u|^{p-2} uu_{\varepsilon}   }{\delta ^p}\dx .
\end{equation}
Finally, passing to the limit in (\ref{agmonweak2}) as $\varepsilon \to 0$, we get that
\begin{equation}
\label{agmonweak3}
\int_{\Omega }|\nabla u|^{p} \dx \le H_p(\Omega )\int_{\Omega }\frac{|u|^{p} }{\delta ^p}\dx <\infty  .
\end{equation}
as required. Thus $u\in W^{1,p}_0(\Omega)$  since the  Sobolev norm of $u$ is finite and $u$ vanishes at the boundary of $\Omega$.

\medskip

In order to prove estimate (\ref{mainboundedeq}) we proceed as follows. Let $\alpha $ be as in the statement and let $\beta \in (0,1)$ be such
that $\alpha <\beta  <\alpha +\tilde \gamma$. Then we can apply Lemma~\ref{subsol} and conclude that in a suitable relative neighbourhood of $\partial \Omega $
 \begin{eqnarray}
   \left(-\Delta_p-H_p(\Omega) \frac{{\mathcal{I}}_p}{\delta^p}\right)(G^{\alpha}-G^{\beta}) =   \left(-\Delta_p-\lambda_{\alpha } \frac{{\mathcal{I}}_p}{\delta^p}\right)(G^{\alpha}-G^{\beta})  \geq 0.
\end{eqnarray}
Thus $G^{\alpha}-G^{\beta }$ is a positive supersolution to the equation $-\Delta_pv-H_p(\Omega )\frac{{\mathcal{I}}_pv}{\delta^p}=0$ in a relative neighbourhood of $\partial \Omega $. Since $u$ is a positive solution of minimal growth in a neighbourhood of infinity in $\Omega$, it follows that $u$ satisfies $u\le C (G^{\alpha}-G^{\beta})$ in a relative neighbourhood of $\partial \Omega$ for a suitable positive constant $C$. Since $G(x)$ is asymptotic to $\delta (x)$ as $x\to \partial \Omega$, we deduce the validity of   (\ref{mainboundedeq}).
\end{proof}
%%%%%%%%%%%%%%%%
\begin{rem}\label{remweak} {\em In the proof of Theorem~\ref{mainbounded}, the assumption  $ \Omega\in C^{1,\gamma}$ was used in a substantial way only to prove the validity of (\ref{mainboundedeq}), and to establish the upper bound
$\lambda_{p,\infty }(\Omega)\le c_p$.  Note  that $\lambda_{p,\infty }(\Omega)\le c_p$ holds provided there exists one point  $z \in \partial \Gw$ which admits a tangent hyperplane in the sense of \cite[Theorem~5]{mamipi}.

On the other hand, the proof of inequality $\lambda_{p,\infty}(\Omega) \geq c_p$ and  the proof of the existence of a minimizer in $W^{1,p}_0(\Omega )$ under the condition  $H_p(\Omega)<c_p$,  rely only on the assumption that $\Omega $ is of class $ C^1$ and on the  existence of a $p$-harmonic function $u$ defined in a relative neighbourhood of $\partial \Omega$ such that $u(x)=0$ and $\nabla u(x)\ne 0$ for all $x\in \partial \Omega$. Under these weaker assumptions, it was also proved that a slightly weaker estimate holds for the positive minimizer $u$. Namely, estimate (\ref{mainboundedeq}) holds with the power $\alpha $ replaced by any power $\tilde\alpha$ smaller than $\alpha$.  We recall in particular that the condition $\nabla u(x)\ne 0$ for all $x\in \partial \Omega $ is guaranteed by the Hopf lemma which holds under weaker assumptions on $\partial \Omega $, for example under the assumption that
$ \Omega$ is of class  $C^{1, {\rm Dini}}$, see \cite{miksha}.  Recall also that the Hopf lemma does not hold in general  under the sole assumption that $\Omega$ is of class $C^1$, see e.g., \cite[\S~3.2]{gitr}.
 }
\end{rem}

Following the observations of the previous remark, we can state the following variant of the previous theorem.
%%%%%%%%%%%%%
\begin{thm}\label{mainboundedweak} Let $\Omega $ be a bounded domain in $\R^n$ of class $C^{1}$ such that the Hopf lemma holds for the $p$-Laplacian (namely,  any positive $p$-harmonic function $u$ defined in a relative neighbourhood of $\partial \Omega$ such that    $u=0$ on $\partial \Omega$ satisfies $\nabla u(x)\ne 0$ for all $x\in \partial \Omega$).      Then
$\lambda_{p,\infty}(\Omega)\geq c_p$. Moreover, if  $H_p(\Omega )< c_p$,
then there exists a positive minimizer $u\in W^{1,p}_0(\Omega )$ for (\ref{ray}). In particular, if $\alpha \in \, ](p-1)/p, 1[$ is such that $\lambda_{\alpha}=H_p(\Omega )$, then
for any $\varepsilon>0$ there exists $C_{\varepsilon}>0$  such that
$$0<u(x)\le C_{\varepsilon }\delta^{\alpha -\varepsilon} (x) \qquad \forall x\in \Gw. $$
\end{thm}
%%%%%%%%%%%%%%%
We can also consider the case of exterior domains. Recall that  $c^*_{p,n}=|\frac{p-n}{p}|^p$ and $c_{p,n}=\min\{c_p, c^*_{p,n} \}$.  It is well known that
if $p=n$, then $H_p(\Gw)=\lambda_{p,\infty}(\Omega)=0$ \cite{mamipi}. Therefore, in the following theorem we consider the case $p\neq n$.
%%%%%%%%%%%%%%%%%%%%%
\begin{thm}\label{mainexterior} Let  $\Omega $ be an unbounded domain in $\R^n$ with nonempty compact boundary of class $C^{1,\gamma}$ with $\gamma \in ]0,1]$, and let $p\neq n$. Then $\lambda_{p,\infty}(\Omega)= c_{p,n}$.

Moreover,  if  $H_p(\Omega )<  c_{p,n}$,
then there exists a positive  minimizer $u\in \widetilde W^{1,p}(\Omega )$  for (\ref{ray}).  Finally,  let $\alpha , \alpha_1 \in \, ](p-1)/p, 1[$ and $\alpha_2 \in\, ]0,(p-1)/p[$ be  such that $\lambda_{\alpha}=H_p(\Omega)$,  $ \lambda_{\alpha_1}=\lambda_{\alpha_2 }=|(p-1)/(p-n)   |^p H_p(\Omega )$.  Then there exists
$C>0$, an open neighbourhood ${\mathcal{U}}$ of $\partial\Omega $, and $M>0$  such that $u$ satisfies the following estimates:
\begin{itemize}
\item[(i)] $ 0<u(x)\le C\delta^{\alpha} (x)$  for all  $x\in  \Omega \cap  {\mathcal{U}}$.
\item[(ii)] If $p<n$, then $0<u(x)\le C|x|^{\frac{\alpha_1(p-n)}{p-1}}$ for all $|x|>M$.
\item[(iii)] If $p>n$, then  $0<u(x)\le C|x|^{\frac{\alpha_2(p-n)}{p-1}}$ for all $|x|>M$.
\end{itemize}
\end{thm}
%%%%%%%%%%%%%%%%%
\begin{proof}  Let  $\tilde x_0\in \Omega$ and  let $G$ be a positive function defined on $\Omega$ which coincides in a relative neighbourhood of $\partial \Omega$ with the positive minimal  Green function in $\Gw$ of the $p$-Laplacian with pole at $\tilde x_0$, and such that $G(x)=|x|^{\frac{p-n}{p-1}}$ for all $x$ in a neighbourhood of $\infty $ (note that the specific definition  of $G$ outside such neighbourhoods is irrelevant here).

 Since $G$ satisfies (\ref{general1}) for all $x_0\in \partial\Omega$, we can apply the same argument as in the proof of Theorem~\ref{mainbounded}
to conclude that for any $\alpha\in (0,1)$ and  $\varepsilon>0$ sufficiently small the function $G^{\alpha}$ is a positive supersolution
to the equation
\begin{equation}\label{mainext1}
-\Delta_pv-(\lambda_{\alpha} -\varepsilon  ){\mathcal{I}}_pv/\delta^p=0
\end{equation}
in a neighbourhood of $\partial \Omega$.

We note now  that $\Delta_pG=0$ also in  a neighbourhood of $\infty$ and that $G$ satisfies (\ref{devpinlem1}) with $x_0=\infty $ and $c= |p-n|/(p-1)$.
Thus, by (\ref{asymlem1}) it follows that for any $\alpha\in (0,1)$ and $\varepsilon>0$ sufficiently small the function $G^{\alpha}$ is a supersolution
to the equation
\begin{equation}\label{mainext2} -\Delta_pv-\left(  \left|\frac{p-n}{p-1}\right|^p \lambda_{\alpha}  -\varepsilon  \right)\frac{{\mathcal{I}}_pv}{ \delta^p}=0
\end{equation}
in a neighbourhood of $\infty$.

Recall that for  $\alpha =(p-1)/p$ we have  $\lambda_{\alpha } =c_p$, hence $ \left|\frac{p-n}{p-1}\right|^p \lambda_{\alpha}= c^*_{p,n}$. Thus, choosing $\alpha=(p-1)/p$ and looking at the equations (\ref{mainext1}) and  (\ref{mainext2}) we immediately see that for any $\varepsilon >0$ sufficiently small the function
$G^{(p-1)/p}$ is a supersolution of equation $  -\Delta_pv-(c_{p,n} -\varepsilon  ){\mathcal{I}}_pv/\delta^p=0$ in a relative neighbourhood of $\partial \Omega\cup \{\infty \} $. Thus, passing to the limit as $\varepsilon \to 0$ we conclude that $\lambda_{p,\infty}(\Omega)\geq c_{p,n}$.

As in the proof of Theorem~\ref{mainbounded}, we can use the argument of \cite[Theorem~5]{mamipi} in a relative neighbourhood of any point of $\partial \Omega$  to prove that $\lambda_{p, \infty }\le c_p$. Moreover, by \cite[Example~2]{mamipi} it also follows that $\lambda_{p,\infty}(\Omega)\le  c_{n,p}^*$. So, $\lambda_{p,\infty}(\Omega)\le c_{p,n}$. Thus, $\lambda_{p,\infty}(\Omega)= c_{p,n}$.

\medskip

Assume now that $H_p(\Omega )< c_{p,n}$.We need to prove the existence of a minimizer for (\ref{ray}).
As in the proof of Theorem~\ref{mainbounded},  since $H_p(\Omega)<\lambda_{p,\infty }(\Omega)$, by Lemma~\ref{pinchcriterion} it follows that the positive function of minimal growth $u \in {\mathcal{M}}^{Q_{-H_p(\Omega )\delta^{-p}}}_{\Omega , \{\tilde x_0\}}$   is an Agmon ground state.   Arguing as in the proof of Theorem~\ref{mainbounded} we choose $\beta \in (\alpha ,1) $ such that      $\beta <\alpha +\tilde \gamma$ where $\tilde \gamma\in (0,\gamma )$ is such that $G$ is of class $C^{1,\tilde \gamma}$ in a relative neighbourhood of $\partial \Omega$.   Exactly as in the proof of Theorem ~\ref{mainbounded}, it turns out that $G^{\alpha}-G^{\beta}$ is a positive supersolution to the equation $ -\Delta_pv-H_p(\Omega )\frac{{\mathcal{I}}_pv}{ \delta^p}=0    $ in a  relative neighbourhood of $\partial \Omega $. Hence,
the Agmon ground state $u$ satisfies the condition $u \le C (G^{\alpha}-G^{\beta })$  in a relative neighbourhood of $\partial \Omega $  which provides the validity of the estimate in statement (i) for the function $u$.

\medskip

In  order to analyze the behaviour of $u$ at $\infty $, we use Lemma \ref{subsolinf}. We consider first the case $p<n$. Let $\beta \in (0,1)$ be such $\alpha_1<\beta <\alpha_1 + (p-1)/(n-p)  $. Then by Lemma \ref{subsolinf} we have
$$
-\Delta_p(G^{\alpha_1}-G^{\beta})\geq \lambda_{\alpha_1 }\left| \frac{p-n}{p-1}\right|^p{\mathcal{I}}_p(G^{\alpha_1}-G^{\beta})/\delta^p=
H_p(\Omega ){\mathcal{I}}_p(G^{\alpha_1}-G^{\beta})/\delta^p
$$
in a neighbourhood of $\infty $, which means that  $G^{\alpha_1}-G^{\beta}$ is a supersolution. Thus  $u$ satisfies the condition
$u(x)\le C (G^{\alpha_1}-G^{\beta}  )$, which implies that $u$ satisfies the estimate in statement (ii) in a neighbourhood of $\infty$ (note that for $p<n$, $G(x)\to 0$ as $|x|\to \infty$, hence the leading term in $G^{\alpha_1}-G^{\beta }$ is given by $G^{\alpha_1}$).

As far as the case $p>n$ we argue in the same way. We consider $\beta \in (0,1)$ such that $0<\beta <\alpha _2$ and we get that
$G^{\alpha _2}-G^{\beta }$ is a supersolution in a neighbourhood of $\infty$. Thus the Agmon ground state $u$ satisfies the estimate in statement (iii) in a neighbourhood of $\infty$ (note that for $p>n$, $G(x)\to \infty $ as $|x|\to \infty$, hence the leading term in $G^{\alpha_2}-G^{\beta }$ is given by $G^{\alpha_2}$).

In conclusion, we have proved that  $u$ satisfies the appropriate estimates in statements (i), (ii), (iii). This implies that
$u\in L^p(\Omega; \delta ^{-p})$.

It remains to prove that $\nabla u\in L^p(\Omega )$. We can apply the same argument used in the proof of Theorem~\ref{mainbounded}  to conclude that  $\nabla u\in L^p(U )$, where $U$ is a relative neighbourhood of $\partial \Gw$. On the other hand, since the operator $-\Delta_p - \frac{H_p(\Omega) }{\delta^p} {\mathcal I}_p$ has a Fuchsian type singularity at infinity, it  follows from \cite[Lemma 2.6]{frpi} that there exists $r_0>0$ such that
\begin{equation}\label{grad}
|\nabla u(x)| \leq C\frac{u(x)}{|x|} \qquad \mbox{ for all }  |x|>r_0.
\end{equation}
Since $u\in L^p(\Omega; \delta ^{-p})$, it follows from \eqref{grad} that $\nabla u\in L^p(\Omega )$.
\end{proof}
%%%%%%%%%%%%%%%%%%
\section{Lower bounds and non-existence of minimizers}\label{sec5}
In the present section we prove that the existence of a minimizer to the variational problem implies the existence of a spectral gap (equivalently, the absence of a spectral gap implies the non-existence of minimizers).  In the case of a bounded domain, the proof is based on a construction of a suitable subsolution for the  equation $-\Delta_pv-c_p\delta^{-p}{\mathcal{I}}_pv=0$ and a comparison principle  proved in \cite[Proposition~3.1]{marsha}. In the case of unbounded domains, the proof is also based on the use of positive solutions  of minimal growth at infinity for equations of the type $-\Delta_pv-\gl |x|^{-p}{\mathcal{I}}_pv=0$.
\begin{thm}\label{mainlower} Let $\Omega $ be a bounded domain in $\R^n$ of class $C^{1,\gamma}$ with $\gamma \in ]0,1]$  and  fix $0<\lambda \le H_p(\Omega)$.   Let $\alpha \in [(p-1)/p ,1 )$ be such that $\lambda_{\alpha}=\lambda $, and let ${\mathcal{U}}$ be an open neighbourhood of $\partial\Omega$.
If $u\in C({\overline{\Omega \cap \mathcal{U}}}\setminus \partial \Gw)$ is a positive solution of the equation
\begin{equation}\label{lambda_alpha}
    -\Delta _pv- \frac{\lambda }{\delta^p}{\mathcal{I}}_pv=0
\end{equation}
in $\Omega \cap \mathcal{U}$, then there exists a constant $C>0$ such that
\begin{equation}\label{mainlower0}
\delta (x) ^{\alpha }\le C u(x) \qquad \mbox {in } \Omega \cap  {\mathcal{U}}.
\end{equation}
Hence, if $u$ is a minimizer in (\ref{ray}), then $H_p(\Omega )< c_p$.
\end{thm}
%%%%%%%%%%%%%%%%%
\begin{proof}  Let $x_0\in \Omega$ be fixed, and let $G$ be the positive minimal Green function in $\Gw$ for the $p$-Laplacian with pole at $x_0$, and note that $G$ vanishes at $\partial \Omega$.
Recall that since  $\Omega $ is of class $C^{1,\gamma}$, by standard regularity theory there exists an open neighbourhood $U_0$ of $\bar{\Omega} $ and $\tilde \gamma \in ]0,\gamma ] $ such that $G$ is of class $C^{1,\tilde \gamma }(\overline{ \Omega \cap U_0})$.  Moreover, since $\Omega $ is of class $C^{1,\gamma}$, the Hopf lemma holds, and hence,  $\nabla G(x)\ne 0$ for all $x\in \partial \Omega$.  Fix $\beta \in (0,1)$ such that $\alpha  <\beta <(p-1)/p+\tilde \gamma$. In light of Lemma~\ref{subsol},  there exist $\vge>0$ and  an open neighbourhood  $U\subset U_0$ of $\partial \Omega $  such that for all $\alpha<\tilde \alpha<\ga+\vge<\gb $, the function $v:=G^{\tilde \alpha }+G^{\beta }$ satisfies
$
-\Delta _pv\le \frac{\lambda_{\tilde \alpha }}{\delta^p}{\mathcal{I}}_pv$ in $ \Omega\cap U$. Hence,
$$
-\Delta _pv\le \frac{\lambda_{\alpha}}{\delta^p}{\mathcal{I}}_pv \qquad {\rm in }\ \Omega\cap U.
$$
Let $C$ be a positive constant such that $v\le C u  $ on $\Omega \cap \partial U$ for all $\tilde \alpha\in (0,1) $  sufficiently close to $\alpha $. Then by the comparison principle proved in \cite[Proposition~3.1]{marsha}, we can conclude that
\begin{equation}
\label{lower1}
v\le C u,\ \   {\rm in }\ \Omega\cap U
\end{equation}
provided
\begin{equation}\label{marshacon}
\liminf_{r\to 0} \frac{1}{r} \int_{D_r}v^p\left( \left|\frac{\nabla v}{v}\right|^{p-1}+ \left|\frac{\nabla u}{u}\right|^{p-1} \right)\dx=0,
\end{equation}
where $D_r=\{ x\in \Omega : r/2<\delta (x)<r \}$.  Since $\tilde \alpha >(p-1)/p$,  condition (\ref{marshacon}) can be verified exactly as in the proof of \cite[Lemma~5.1]{marsha}, where $v$ is replaced by $\delta^{\tilde \alpha }+\delta ^{\beta}$: for  this purpose, note in particular that $G(x)$ is asymptotic to $\delta(x) $ as $x\to \partial \Omega$ and that   \cite[Proposition~2.1~(ii)]{marsha}  holds true also in the case of $C^{1,\gamma}$  domains (and actually also in the case of  $C^{0,1}$ domains) as it can be easily verified. 

Since the constant $C$ in (\ref{lower1})  does not depend on $\tilde \alpha $ for $\tilde \alpha $ close enough to $\alpha$, it follows that
$$
G^{\alpha }\le Cu
$$
in a relative neighbourhood of $\partial \Omega$, by which we can immediately deduce (\ref{mainlower0}) (here one should take $C$ sufficiently large in order to control the function $G^{\alpha}$ not only in a small relative neighbourhood of $\partial \Omega $ but also in the whole of $\Omega \cap \mathcal{U}$). If $H_p(\Omega)=c_p$, then $\alpha =(p-1)/p$. Therefore,  (\ref{mainlower0}) clearly implies that $u\notin W^{1,p}_0(\Omega )$, hence $u$ cannot be a minimizer.
\end{proof}
%%%%%%%%%%%%%%%%%%%%%%%%%%%%%%%%%%%%
Combining the results of Lemma~\ref{subsol} and Theorem~\ref{mainlower} we obtain the following tight upper and lower bounds for positive solutions of minimal growth near $\partial \Gw$.
\begin{corol}\label{cor1}
Let $\Omega $ be a bounded domain in $\R^n$ of class $C^{1,\gamma}$ with $\gamma \in ]0,1]$  and  fix $0<\lambda \le c_p$.   Let $\alpha \in [(p-1)/p ,1 )$ be such that $\lambda_{\alpha}=\lambda $, and let ${\mathcal{U}}$ be an open neighbourhood of $\partial\Omega$. Let $u\in C({\overline{\Omega \cap \mathcal{U}}}\setminus \partial \Gw)$ be a positive solution of the equation
\begin{equation}\label{lambda_alpha2}
    -\Delta _pv- \frac{\lambda }{\delta^p}{\mathcal{I}}_pv=0 \qquad \mbox {in }\ \Omega\cap {\mathcal{U}}
\end{equation}
of minimal growth in a neighbourhood of infinity in $\Gw$.

Then there exists a constant $C>0$ such that
\begin{equation}\label{mainlower02}
C^{-1}\delta (x) ^{\alpha }\le  u(x) \leq C\delta (x) ^{\alpha } \qquad \mbox {in }\ \Omega\cap {\mathcal{U}}.
\end{equation}
\end{corol}
\medskip
\begin{rem}{\em  
In the limiting case $\lambda =0$, and under the mild regularity assumptions of Lemma~\ref{general}, one obtains from estimate \eqref{general2}  that \eqref{mainlower02}  holds with the limiting  exponent $\alpha =1$. This gives a strong indication
to our feeling that in general, estimate \eqref{mainlower02} does not hold if the bounded domain $\Gw$ is merely in the $C^{1}$ class. 
Indeed, note that \eqref{mainlower02} does not hold in the class of bounded {\em Lipschitz domains}.  Indeed, for $n=2$, $p=2$, $\lambda =0$ one can take
the {\em Lipschitz domain} $\Omega=]0,1[\times ]0,1[$, and note that the function $u(x,y)=xy$ is a harmonic function of minimal growth near $(0,0)$ that does not satisfy estimate  \eqref{mainlower02} with $\alpha =1$. For other examples concerning the case of a general conic point, $p=2$ and $0<\gl\leq c_2$, see \cite{depips}. 
 }
\end{rem}
Next, we prove that for  a $C^{1,\gamma}$-exterior domain, the existence of a minimizer to the variational problem implies the existence of a spectral gap.
\begin{thm}\label{mainlower1}
Assume that $\Omega $ is an unbounded domain in $\R^n$ with nonempty compact boundary of class $C^{1,\gamma}$ with $\gamma \in ]0,1]$. Fix $p\neq n$ and $0<\lambda\leq H_p(\Omega)$.
Let $\alpha , \alpha_1 \in \, [(p-1)/p, 1[$ and $\alpha_2 \in\, ]0,(p-1)/p]$ be  such that  $\lambda=\lambda_{\alpha}:=(p-1)\alpha ^{p-1}(1-\alpha)$, $\lambda_{\alpha_1}:=\lambda_{\alpha_2 }=|(p-1)/(p-n)   |^p \lambda $.

 If   $u$ is a positive solution of the equation (\ref{lambda_alpha}), then there exists
$C>0$, an open neighbourhood ${\mathcal{U}}$ of $\partial\Omega $ and $M>0$  such that $u$ satisfies the following estimates:
\begin{itemize}
\item[(i)] $ u(x)\geq C\delta^{\alpha} (x)$  for all  $x\in  \Omega  \cap {\mathcal{U}}$.
\item[(ii)] If $p<n$, then $u(x)\geq C|x|^{\frac{\alpha_1(p-n)}{p-1}}$ for all $|x|>M$.
\item[(iii)] If $p>n$, then  $u(x)\geq C|x|^{\frac{\alpha_2(p-n)}{p-1}}$ for all $|x|>M$.
\end{itemize}
Hence,  if $u$ is a minimizer for  (\ref{ray}), then   $H_p(\Omega )< c_{p,n}$.
\end{thm}
%%%%%%%%%%%%%%%%%%%%%%%%%%%%%%%%%%%%%
\begin{proof} Without loss of generality, we may assume that $0\in \R^n\setminus\bar{\Gw}$, and let $R>0$ be such that $\R^n\setminus \Gw\subset B(0,R)$.
Recall that by Theorem~\ref{mainexterior}, $H_p(\Omega )\leq c_{p,n}$.

 Let $u$ be any positive solution of the equation (\ref{lambda_alpha}).
 Since  $H_p(\Omega )\leq c_{p,n}\leq c_p$, estimate (i) follows from the proof of Theorem~\ref{mainlower}.

On the other hand, for any  $0< \gm \leq c^*_{p,n}$ consider the equation
\begin{equation}\label{eq_h3}
  -\Delta _pv- \frac{\gm}{|x|^p}{\mathcal{I}}_pv=0
\end{equation}
in $\R^n\setminus \{0\}$.
Then $v_\gm(x):=|x|^{\gb(\gm)}$ is a positive solution of \eqref{eq_h3} of minimal growth near $\infty$, where $\gb(\gm)\leq (p-n)/p$ is the larger  (resp., smaller) root if $p<n$ (resp., if $p>n$)  of the transcendental equation
$$-\gb |\gb|^{p-2}[\gb(p-1)+n-p]=\gm,$$
see, \cite[Example 1.1]{frpi11}.
Note that $\gb(\gm)= (p-n)/p$ if and only if $\mu = c^*_{p,n}$. Note also that  $\beta (\mu )= \alpha (\mu) (p-n)/(p-1)$ where $\alpha (\mu )$ is the  is the larger  (resp., smaller) positive real number     such that
$\lambda_{\alpha(\mu ) }= |(p-1)/(p-n)|^p \mu $ if $p<n$ (resp., if $p>n$).

 Take,  $\gm= \lambda $. Then $u$ is a positive supersolution of \eqref{eq_h3} in
$\R^n\setminus  B(0,R)$  since $u$ is a solution of (\ref{lambda_alpha}) and $\delta (x)\le |x|$ for all $x\in \R^n\setminus  B(0,R)$.  Therefore, there exists a positive constant $c$ such that  $c|x|^{\gb(\gm)}\leq u(x)$ in $\R^n\setminus  B(0,R)$, and we obtain estimate (ii) if $p<n$ and (iii) if $p>n$.

Estimates (i)-(iii) for $u$ clearly imply that if  $\lambda =H_p(\Omega )=c_{p,n}$,  then any positive solution $u$ satisfies $u\notin W^{1,p}_0(\Omega )$, and therefore, the variational problem does not admit a minimizer.
\end{proof}
Finally, by combining the results of Lemma~\ref{subsolinf}, Corollary~\ref{cor1}, and Theorem~\ref{mainlower1}, we obtain for $C^{1,\gamma}$-exterior domains $\Gw$ tight upper and lower bounds for positive solutions of minimal growth in a neighbourhood of infinity in $\Gw$, 
that is, a neighbourhood of $\partial \Omega\cup \{\infty\}$.

\begin{corol}\label{cor2}
Let  $\Omega $ be an unbounded domain in $\R^n$ with nonempty compact boundary of class $C^{1,\gamma}$ with $\gamma \in ]0,1]$, and fix $p\neq n$ and $0<\lambda\leq c_{p,n}$.
Let $\alpha , \alpha_1 \in \, [(p-1)/p, 1[$ and $\alpha_2 \in\, ]0,(p-1)/p]$ be  such that  $\lambda=\lambda_{\alpha}=(p-1)\alpha ^{p-1}(1-\alpha)$, $\lambda_{\alpha_1}=\lambda_{\alpha_2 }=|(p-1)/(p-n)   |^p \lambda $.
 If   $u$ is a positive solution of the equation (\ref{lambda_alpha}) of minimal growth in a neighbourhood of infinity in $\Gw$,
   then there exists
$C>0$, an open neighbourhood ${\mathcal{U}}$ of $\partial\Omega $ and $M>0$  such that $u$ satisfies the following estimates:
\begin{itemize}
\item[(i)] $ C^{-1}\delta^{\alpha} (x)\leq  u(x)\leq C\delta^{\alpha} (x)$  for all  $x\in   \Omega\cap {\mathcal{U}} $.
\item[(ii)] If $p<n$, then $C^{-1}|x|^{\frac{\alpha_1(p-n)}{p-1}} \leq u(x)\leq C|x|^{\frac{\alpha_1(p-n)}{p-1}}$ for all $|x|>M$.
\item[(iii)] If $p>n$, then  $C^{-1}|x|^{\frac{\alpha_2(p-n)}{p-1}}\leq u(x)\leq C|x|^{\frac{\alpha_2(p-n)}{p-1}}$ for all $|x|>M$.
\end{itemize}
\end{corol}

\medskip

\begin{center}{\bf Acknowledgments} \end{center}

The authors are grateful to the anonymous referee for reading the paper very carefully and pointing out one blunder of technical type which was possible to fix.

This research was initiated in 2015 when the second author visited the Department of Mathematics of the University of Padova in the frame of the {\it Visiting Scientist Program} of the University of Padova.  Both authors acknowledge the warm hospitality and the financial support received by each other's institution on the occasion of their research visits. The  first author is also a member of the Gruppo Nazionale per l'Analisi Ma\-te\-ma\-ti\-ca, la Probabilit\`{a} e le loro Applicazioni (GNAMPA) of the
Istituto Nazionale di Alta Matematica (INdAM).
 This research was also supported by the INDAM - GNAMPA project 2017 ``Equazioni alle derivate parziali non lineari e disuguaglianze funzionali: aspetti geometrici ed analitici". The second author acknowledges the support of the Israel Science Foundation (grant No. 970/15) founded by the Israel Academy of Sciences and Humanities.

\end{document}